\documentclass{amsart}

\usepackage[latin1]{inputenc}
\usepackage{amsfonts}
\usepackage{amsmath}
\usepackage{amsthm}
\usepackage{amssymb}
\usepackage{latexsym}
\usepackage{enumerate}
\usepackage{tikz-cd}
\usepackage{caption}

\newtheorem{theorem}{Theorem}[section]
\newtheorem{lemma}[theorem]{Lemma}
\newtheorem{proposition}[theorem]{Proposition}
\newtheorem{corollary}[theorem]{Corollary}

\newtheorem{remark}[theorem]{Remark}

\newtheorem{question}[theorem]{Question}
\newtheorem{problem}[theorem]{Problem}
\newtheorem*{claim}{Claim}
\newtheorem{construction}[theorem]{Construction}

\newenvironment{proof1} {\begin{proof}[Proof of Proposition \ref{construction-properties}]} {\end{proof}}
\newenvironment{proof2} {\begin{proof}[Proof of the claim]} {\end{proof}}

\newtheorem{definition}[theorem]{Definition}

\theoremstyle{definition}
\newtheoremstyle{sltheorem}
{}                
{}                
{\slshape}        
{}                
{\bfseries}       
{.}               
{ }               
{\thmname{#1}\thmnote{ \bfseries #3}}                
\theoremstyle{sltheorem}

\theoremstyle{sltheorem}
\newtheorem{reptheorem}{Theorem}

\DeclareMathOperator{\lin}{span}

\DeclareMathOperator{\even}{Even}
\DeclareMathOperator{\odd}{Odd}
\DeclareMathOperator{\supp}{supp}
\DeclareMathOperator{\rd}{rd}
\DeclareMathOperator{\graph}{graph}

\begin{document}

\newcommand{\cc}{\mathfrak{c}}
\newcommand{\N}{\mathbb{N}}
\newcommand{\BB}{\mathbb{B}}
\newcommand{\C}{\mathbb{C}}
\newcommand{\Q}{\mathbb{Q}}
\newcommand{\R}{\mathbb{R}}
\newcommand{\Z}{\mathbb{Z}}
\newcommand{\T}{\mathbb{T}}
\newcommand{\PP}{\mathbb{P}}
\newcommand{\rin}{\right\rangle}
\newcommand{\SSS}{\mathbb{S}}
\newcommand{\forces}{\Vdash}
\newcommand{\dom}{\text{dom}}
\newcommand{\osc}{\text{osc}}
\newcommand{\F}{\mathcal{F}}
\newcommand{\A}{\mathcal{A}}
\newcommand{\B}{\mathcal{B}}
\newcommand{\I}{\mathcal{I}}
\newcommand{\X}{\mathcal{X}}
\newcommand{\Y}{\mathcal{Y}}
\newcommand{\CC}{\mathcal{C}}
\newcommand{\non}{\mathfrak{non}}
\newcommand{\add}{\mathfrak{add}}
\newcommand{\cov}{\mathfrak{cov}}
\newcommand{\cof}{\mathfrak{cof}}

\author{Damian G\l odkowski}
\address{Institute of Mathematics of the Polish Academy of Sciences,
ul.  \'Sniadeckich 8,  00-656 Warszawa, Poland}
\address{Faculty of Mathematics, Informatics, and Mechanics, 
University of Warsaw, ul. Banacha 2, 02-097 Warszawa, Poland}
\email{\texttt{d.glodkowski@uw.edu.pl}}

\thanks{This research was funded in part by the NCN (National Science
Centre, Poland) research grant no. 2021/41/N/ST1/03682. For the purpose of Open Access, the author has applied a CC-BY public copyright licence to any
Author Accepted Manuscript (AAM) version arising from this submission.}


\title{A Banach space C(K) reading the dimension of K}

\begin{abstract} 
Assuming Jensen's diamond principle ($\diamondsuit$) we construct for every natural number $n>0$ a compact Hausdorff space $K$ such that whenever the Banach spaces $C(K)$ and $C(L)$ are isomorphic for some compact Hausdorff $L$, then the covering dimension of $L$ is equal to $n$. The constructed space $K$ is separable and connected, and the Banach space $C(K)$ has few operators i.e. every bounded linear operator $T:C(K)\rightarrow C(K)$ is of the form $T(f)=fg+S(f)$, where $g\in C(K)$ and $S$ is weakly compact.
\end{abstract}


\maketitle
\section{Introduction}
In \cite{few-operators-main} Koszmider showed that there is a compact Hausdorff space $K$ such that whenever $L$ is compact Hausdorff and the Banach spaces $C(K)$ and $C(L)$ are isomorphic, the dimension of $L$ is greater than zero. In the light of this result Pe\l{}czy\'{n}ski asked, whether there is a compact space $K$ with $\dim(K)=k$ for given $k\in \omega\backslash\{0\}$, such that if $C(K)\sim C(L)$, then $\dim(L)\geq k$ (\cite[Problem 4]{few-operators-survey}). We show that the answer to this question is positive, if we assume Jensen's diamond principle ($\diamondsuit$). Namely, we prove the following:
 \begin{reptheorem}[\ref{main}]
Assume $\diamondsuit$. Then for every $k\in \omega\cup \{\infty\}$ there is a compact Hausdorff space $K$ such that $\dim(K)=k$ and whenever $C(K)\sim C(L)$, $\dim(L)=k$.
 \end{reptheorem}

Note that typically dimension of $K$ is not an invariant of the Banach space $C(K)$ under isomorphisms. For instance, the classical result by Miljutin says that if $K,L$ are compact metrizable uncountable spaces, then the Banach spaces $C(K)$ and $C(L)$ are isomorphic (\cite{miljutin}). This also shows that $C(K)$ with the desired property cannot admit any complemented copy of $C(L)$ where $L$ is compact, metrizable and uncountable (indeed, if $C(K)\sim X\oplus C(L)$, then $C(K)\sim X\oplus C(L)\oplus C([0,1]^n) \sim C(K)\oplus C([0,1]^n)$ for any $n\in \omega$). Another result by Pe\l{}czy\'{n}ski says that if $G$ is an infinite compact topological group of weight $\kappa$, then $C(G)$ is isomorphic to $C(\{0,1\}^\kappa)$ (\cite{pelczynski}). 

On the other hand the space $C(K)$ remembers many topological and set-theoretic properties of $K$. For example Cengiz showed that if $C(K)\sim C(L)$, then $K$ and $L$ have the same cardinalities (\cite{cengiz}). If $K$ is scattered, then by Pe\l{}czy\'{n}ski-Semadeni theorem $L$ is scattered as well (\cite{pelczynski-semadeni}). In this case both spaces must be zero-dimensional. If $K$ is an Eberlein compact, then $L$ is also Eberlein (\cite{negrepontis-banach-spaces}). If $K$ is a Corson compact and $L$ is homogeneous, then $L$ is Corson (\cite{plebanek-corson}). 

Although the isomorphic structure of $C(K)$ does not remember the dimension of $K$, the metric structure of $C(K)$ contains such information, since by the Banach-Stone theorem $K$ and $L$ are homeomorphic, whenever $C(K)$ and $C(L)$ are isometric. Similar results were obtained by Gelfand, Kolmogorov and Kaplansky in the category of rings of functions on compact spaces and in the category of Banach lattices (\cite{gelfand-kolmogorov, kaplansky}). It is also worth to mention that the covering dimension of $K$ is an invariant for the space $C_p(K)$ of continuous functions on $K$ with the pointwise topology (\cite{pestov}).

The key property of the space $K$ that we construct to prove Theorem \ref{main} is the fact that the Banach space $C(K)$ has few operators i.e. every bounded operator $T:C(K)\rightarrow C(K)$ is of the from $T=gI+S$, where $g\in C(K)$ and $S$ is weakly compact. Schlackow showed that if the Banach space $C(K)$ has few operators, $C(K)\sim C(L)$ and both spaces $K,L$ are perfect, then $K$ and $L$ are homeomorphic (\cite{schlackow}). We improve this result under the assumption that $K$ is separable and connected. 

 \begin{reptheorem}[\ref{few-operators-modulo-finite}]
Suppose that $K$ is a separable connected compact Hausdorff space such that $C(K)$ has few operators and $L$ is a compact Hausdorff space such that $C(K)\sim C(L)$. 

Then $K$ and $L$ are homeomorphic modulo finite set i.e. there are open subsets $U\subseteq K, V\subseteq L$ and finite sets $E\subseteq K, F\subseteq L$ such that $U, V$ are homeomorphic and $K=U\cup E, L=V\cup F$.
 \end{reptheorem}

 The first example (under the continuum hypothesis) of a Banach space $C(K)$ with few operators appeared in the work of Koszmider (\cite{few-operators-main}). Later, Plebanek showed how to remove the use of {\sf CH} from such constructions (\cite{plebanek-few-operators}). Considered spaces have many interesting properties (cf. \cite[Theorem 13]{few-operators-survey}) e.g. $C(K)$ is indecomposable Banach space, it is not isomorphic to any of its proper subspaces nor any proper quotient, it is a Grothendieck space, $K$ is strongly rigid (i.e. identity and constant functions are the only continuous functions on $K$) and does not include non-trivial convergent sequences. For more examples and properties of Banach spaces $C(K)$ with few operators see \cite{BF2, fajardo, big-densities, few-operators-survey, KMM, koszmider-shelah}.

In the further part of the paper we show how to construct a Banach space $C(K)$ with few operators, where $K$ has arbitrarily given dimension. Theorem \ref{main} is an almost immediate consequence of Theorem \ref{few-operators-modulo-finite} and the following theorem. 

 \begin{reptheorem}[\ref{theorem-construction}]
 Assume $\diamondsuit$.
For each $k>0$ there is a compact Hausdorff, separable, connected space $K$ such that $C(K)$ has few operators and $\dim K=k$.
 \end{reptheorem}
 
 Our construction is a modification of one of the spaces $K$ from \cite[Theorem 6.1]{few-operators-main}, which is a separable connected compact space such that $C(K)$ has few operators. The original space is constructed as an inverse limit of metrizable compact spaces $(K_\alpha)_{\alpha<\omega_1}$, where on intermediate steps we add suprema to countable families of functions in the lattice $C(K_\alpha)$ for $\alpha<\omega_1$, using the notion of strong extension. However, the considered families of functions are very general, which leads to the problem that described operation may rise the dimension of given space and the final space is infinite-dimensional. We show that under $\diamondsuit$ we are able to limit the choice of functions in the way that we can control the dimension of the spaces at each step. In order to control the dimension we introduce the notion of essential-preserving maps. Similar ideas were studied in Fedorchuk's work (\cite{fedorchuk1, fedorchuk2, fedorchuk3}). For instance, Fedorchuk considered maps that are ring-like, monotonic and surjective, which implies that they are essential-preserving (however, those notions are much stronger and are not applicable in our context). 

One may also consider other notions of dimension such as small or large inductive dimension. However, since Theorem \ref{inverse-dimension} does not work if we replace the covering dimension with one of the inductive dimensions, we do not know if the spaces we constructed have finite inductive dimensions. 

The structure of the paper is as follows. Section \ref{notation} concerns basic terminology. Section \ref{section-covering-dimension} contains necessary results about covering dimension. In section \ref{section-few-operators} we prove Theorem \ref{few-operators-modulo-finite} characterizing properties of spaces $C(K)$ with few operators preserved under isomorphisms. In Section \ref{section-extensions} we develop tools for controlling dimension in some inverse limits of systems of compact spaces. Section \ref{section-construction} contains the description of the construction leading to the main theorem of the paper. The last section includes remarks and open questions.  

\section{Notation and terminology}\label{notation}
Most of notation that we use should be standard. For unexplained terminology check \cite{dimension-theory, banach-space-theory, jech}. 
$\omega$ denotes the set of non-negative integers, which is also the smallest infinite ordinal number. $\omega_1$ is the smallest uncountable ordinal. $Lim$ stands for the class of all limit ordinals. $\odd$ and $\even$ stand for the classes of odd and even ordinals respectively. If $f$ is a function, then $f|A$ denotes the restriction of $f$ to $A$. $\sum_{n\in\omega} f_n$ will always denote the pointwise sum of functions $f_n$ (if the sum exists). $[A]^{<\omega}$ is the family of all finite subsets of $A$. For a topological space $X$, $\dim X$ denotes the covering dimension (also known as Lebesgue covering dimension or topological dimension, cf. \cite[Definition 1.6.7]{dimension-theory}) of $X$. $X'$ stands for the subset of $X$ consisting of non-isolated points in $X$. A sequence $(x_n)_{n\in\omega}$ is said to be non-trivial, if it is not eventually constant. We say that a topological space $X$ is c.c.c. if every family of pairwise disjoint open subsets of $X$ is countable. By basic open subset of $[0,1]^{\omega_1}$ we mean a product  $\prod_{\alpha<\omega_1} U_\alpha$ where each $U_\alpha\subseteq [0,1]$ is a relatively open interval with rational endpoints and $U_\alpha=[0,1]$ for all but finitely many $\alpha$'s. A subset $S\subseteq \omega_1$ is called stationary, if it has non-empty intersection with every closed and unbounded subsets of $\omega_1$. 

All considered topological spaces are Hausdorff. We work with Banach spaces of the form $C(K)$ consisting of real-valued continuous functions on a compact space $K$ equipped with the supremum norm. $C_I(K)$ denotes the subset of $C(K)$ of functions with the range included in $[0,1]$. For Banach spaces $X$ and $Y$, a bounded linear operator $T:X\rightarrow Y$ is said to be weakly compact if the closure of $T[B_X]$ is compact in the weak topology in $Y$ (here $B_X$ stands for the unit ball in $X$). $X\sim Y$ means that $X$ and $Y$ are isomorphic as Banach spaces. $\mathcal{B}(X)$ denotes the algebra of all bounded operators on a Banach space $X$ (with the operator norm). An operator $T:C(K)\rightarrow C(L)$ is multiplicative, if $T(fg)=T(f)T(g)$. We will use one symbol $\| \cdot \|$ to denote norms in all considered Banach spaces - this should not lead to misunderstandings.  

{\sf ZFC} stands for Zermelo-Fraenkel set theory with the axiom of choice. {\sf CH} is the continuum hypothesis i.e. the sentence $2^\omega=\omega_1$. Jensen's diamond principle ($\diamondsuit$) stands for the following sentence (for other equivalent formulations see \cite{devlin-diamond}): there is a sequence of sets $A\subseteq \alpha$ for $\alpha<\omega_1$ such that for any subset $A\subseteq \omega_1$ the set $\{\alpha: A\cap \alpha = A_\alpha\}$ is stationary in $\omega_1$. It is a well-known fact, that $\diamondsuit$ implies {\sf CH}.

\subsection*{Radon measures on compact spaces} 
For a compact space $K$ we will identify the space of bounded linear functionals on $C(K)$ with the space of Radon measures on $K$ (the identification is given by the Riesz representation theorem). For every $\alpha<\omega_1$ we have an embedding $E_\alpha: C([0,1]^\alpha) \rightarrow C([0,1]^{\omega_1})$ given by $E_\alpha(f)= f\circ \pi_\alpha$, where $\pi_\alpha:[0,1]^{\omega_1}\rightarrow [0,1]^\alpha$ is the natural projection. For a Radon measure $\mu$ on $[0,1]^{\omega_1}$ we will denote by $\mu|C([0,1]^\alpha)$ the restriction of $\mu$ treated as a functional on $C([0,1]^{\omega_1})$ to the subspace $E_\alpha[C([0,1]^\alpha)]$. Equivalently, $\mu|C([0,1]^\alpha)$ is a measure on $[0,1]^\alpha$ given by $$\mu|C([0,1]^\alpha)(A)=\mu(\pi_\alpha^{-1}(A)).$$
For any measure $\mu$ we denote by $|\mu|$ its variation. 

\section{Covering dimension}\label{section-covering-dimension}
This section is devoted to the basic properties of covering dimension and its behavior in inverse limits of compact spaces. 
We start with several basic definitions. Recall that for a family $\mathcal{A}$ of sets we define its order as the largest integer $n$ such that $\mathcal{A}$ contains $n+1$ sets with non-empty intersection. If there is no such $n$, then we say that the order of $\mathcal{A}$ is $\infty$.  
\begin{definition}\cite[Definition 1.6.7]{dimension-theory}\label{dim}
Let $X$ be a topological space. We say that covering dimension of $X$ (denoted by $\dim X$) is not greater than $n$, if every finite open cover of $X$ has a finite open refinement of order at most $n$.  We say that $\dim X=n$ if $\dim X \leq n$, but not $\dim X \leq n-1$. If there is no $n$ such that $\dim X=n$, then we say that $\dim X=\infty$.  
\end{definition}

\begin{definition}\cite[Definition 1.1.3]{dimension-theory}\label{partition}
Let $X$ be a topological space. A closed set $P\subseteq X$ is a partition between $A$ and $B$ if there are disjoint open sets $U\supseteq A, V\supseteq B$ such that $X\backslash P=U\cup V$.
\end{definition}

\begin{definition}\cite[p. 16]{dimension-theory-char}\label{essential}
A family $\{(A_i,B_i): i=1,2,\dots, n\}$ of pairs of disjoint closed subsets of a space $X$ is called essential if for every family $\{C_i:i=1,2,\dots,n\}$ such that for each $i\leq n$ the set $C_i$ is a partition between $A_i$ and $B_i$ we have $$\bigcap_{i=1}^n C_i\neq \varnothing.$$
\end{definition}

For the proof of the following theorems see \cite[Lemma 3.2, Theorem 3.3]{dimension-theory-char}.

\begin{theorem}\label{essential-char}
For a normal space $X$ the following conditions are equivalent: 
\begin{enumerate}
    \item a family $\{(A_i,B_i): i=1,2,\dots, n\}$ of pairs of disjoint closed sets is not essential in $X$,
    \item for each $i=1,2,\dots n$ there are disjoint open sets $U_i,V_i$ such that $A_i\subseteq U_i, B_i\subseteq V_i$ and $$\bigcup_{i=1}^n (U_i\cup V_i)=X,$$
    \item for each $i=1,2,\dots n$ there are disjoint closed sets $C_i,D_i$ such that $A_i\subseteq C_i, B_i\subseteq D_i$ and $$\bigcup_{i=1}^n (C_i\cup D_i)=X.$$
\end{enumerate}
\end{theorem}

\begin{theorem}\label{dimension-low-bound}
For a normal space $X$ the following conditions are equivalent:
\begin{enumerate}
    \item $\dim X\geq n$,
    \item there is an essential family in $X$ consisting of $n$ pairs.
\end{enumerate}
\end{theorem}

\begin{definition}
Let $\pi:L\rightarrow K$ be a continuous function between compact Hausdorff spaces. We will say that $\pi$ is essential-preserving if for every family $\{(A_i,B_i): i=1,2,\dots, n\}$ essential in $K$, the family $\{(\pi^{-1}(A_i),\pi^{-1}(B_i)): i=1,2,\dots, n\}$ is essential in $L$. 
\end{definition}

Note that Theorem \ref{dimension-low-bound} immediately implies that if $\pi:L\rightarrow K$ is essential-preserving, then $\dim L\geq \dim K$.

\begin{lemma}\label{inverse-closed-sets}\cite[Lemma 16.1]{dimension-theory-char}
Assume that $K_\gamma$ is an inverse limit of a system $\{K_\alpha:\alpha<\gamma\}$, where $K_\alpha$ are compact Hausdorff spaces. If $A,B$ are closed disjoint subsets of $K_\gamma$ then there is $\alpha<\gamma$ such that $\pi^\gamma_\alpha[A], \pi^\gamma_\alpha[B]$ are disjoint subsets of $K_\alpha$, where $\pi^\gamma_\alpha$ stands for the canonical projection from $K_\gamma$ into $K_\alpha$.
\end{lemma}

\begin{theorem}\label{property-E}
Let $\{K_\alpha: \alpha<\gamma\}$ be an inverse system of compact Hausdorff spaces with inverse limit $K_\gamma$ such that for each limit ordinal $\beta<\gamma$, $K_\beta$ is an inverse limit of $\{K_\alpha: \alpha<\beta\}$. Assume that for each $\alpha<\gamma$ the map $\pi^{\alpha+1}_\alpha:K_{\alpha+1}\rightarrow K_\alpha$ is surjective and essential-preserving. Then the canonical projection $\pi^\gamma_1:K_\gamma\rightarrow K_1$ is essential-preserving. In particular $\dim K_\gamma\geq \dim K_1$.   
\end{theorem}

\begin{proof}
We will prove by induction on $\alpha$ that $\pi^\alpha_1:K_\alpha\rightarrow K_1$ is essential-preserving. For successor ordinal $\alpha+1$ it is enough to observe that if $\{(A_i,B_i): i=1,\dots, n\}$ is essential in $K_1$, then $\{((\pi^\alpha_1)^{-1}(A_i),(\pi^\alpha_1)^{-1}(B_i)): i=1,\dots, n\}$ is essential in $K_\alpha$ and hence
 $\{((\pi^{\alpha+1}_1)^{-1}(A_i),(\pi^{\alpha+1}_1)^{-1}(B_i)): i=1,\dots, n\} = \{((\pi^{\alpha+1}_\alpha)^{-1}((\pi^\alpha_1)^{-1}(A_i)),\\ (\pi^{\alpha+1}_\alpha)^{-1}((\pi^\alpha_1)^{-1}(B_i))): i=1,\dots, n\}$ 
is essential in $K_{\alpha+1}$. 

Let $\alpha$ be a limit ordinal and that for each $\beta<\alpha$ the map $\pi^\beta_1:K_\beta \rightarrow K_1$ is essential-preserving. Let $\{(A_i,B_i): i=1,\dots, n\}$ be an essential family in $K_1$ and assume that $\{(\pi^\alpha_1)^{-1}(A_i),(\pi^\alpha_1)^{-1}(B_i)): i=1,\dots, n\}$ is not essential in $K_\alpha$. Then by Theorem \ref{essential-char} for each $i\leq n$ there are closed disjoint sets $C_i\supseteq (\pi^\alpha_1)^{-1}(A_i), D_i\supseteq (\pi^\alpha_1)^{-1}(B_i)$ such that $$\bigcup_{i=1}^n (C_i\cup D_i)=K_\alpha.$$
By Lemma \ref{inverse-closed-sets} for each $i$ there is $\beta_i<\alpha$ such that $\pi^\alpha_{\beta_i}[C_i],\pi^\alpha_{\beta_i}[D_i]$ are disjoint subsets of $K_{\beta_i}$. In particular $\pi^\alpha_\beta[C_i],\pi^\alpha_\beta[D_i]$ are disjoint closed subsets of $K_\beta$, where $\beta=\max\{\beta_i:i\leq n\}$. Since $K_\alpha$ is an inverse limit of surjective maps $\pi^\alpha_\beta$ is also surjective and so 
$$\bigcup_{i=1}^n (\pi^\alpha_\beta[C_i]\cup \pi^\alpha_\beta[D_i])=K_\beta.$$
Moreover, $(\pi^\beta_1)^{-1}(A_i)\subseteq \pi^\alpha_\beta[C_i]$ and $(\pi^\beta_1)^{-1}(B_i)\subseteq \pi^\alpha_\beta[D_i]$, so $\{(\pi^\beta_1)^{-1}(A_i), \pi^\alpha_\beta[C_i] : i\leq n \}$ is not essential in $K_\beta$ which contradicts the inductive assumption. 
\end{proof}

We will need some basic but important properties of the covering dimension.

\begin{theorem}\label{closed-subspace}\cite[Theorem 3.1.3]{dimension-theory}
If $M$ is a closed subspace of a normal space $X$, then $\dim M\leq \dim X$.
\end{theorem}

\begin{theorem}\label{countable-sum-theorem}\cite[Theorem 3.1.8]{dimension-theory}
Let $n\in\omega\cup \{\infty\}$.
If a normal space $X$ is a union of countably many closed subspaces $\{F_i\}_{i\in\omega}$ with $\dim F_i\leq n$, then $\dim X \leq n$.
\end{theorem}

\begin{theorem}\label{product-dim}\cite[Theorem 3.2.13]{dimension-theory}
If $X,Y$ are non-empty compact Hausdorff spaces, then $\dim(X\times Y)\leq \dim X + \dim Y$.
\end{theorem}

\begin{theorem}\label{inverse-dimension}\cite[Theorem 3.4.11]{dimension-theory}
If $K$ is an inverse limit of compact Hausdorff spaces of dimension at most $n$, then $\dim K \leq n$.
\end{theorem}

\begin{definition}\cite[p. 170]{dimension-theory}
Let $A$ be a subspace of a space $X$. We define the relative dimension of $A$ as
$$ \rd_X A =\sup \{ \dim F: F\subseteq A, F \ \text{closed in} \ X \}. $$
\end{definition}

\begin{lemma}\label{fin-points}
Let $n\in\omega\cup \{\infty\}$.
Assume that a normal space $X$ can be represented as a union $U\cup F$ where $F$ is finite and $\rd_X U\leq n$. Then $\dim X \leq n$.  
\end{lemma}

\begin{proof}
This is a special case of \cite[Lemma 3.1.6]{dimension-theory} (which says that if $X=\bigcup_{i=0}^\infty F_i$ and for each $k\in\omega$ the subspace $\bigcup_{i=0}^k F_i$ is closed in $X$, and $\rd_X F_k\leq n$, then $\dim X\leq n$) where $F_0=F, F_1=U$ and $F_n =\varnothing$ for $n>1$.
\end{proof}

\begin{theorem}\label{dim-equal}
Assume that compact Hausdorff spaces $X$ and $Y$ can be represented as $X=U\cup F, Y=V\cup E$ where $U,V$ are open, $E,F$ are finite, $U\cap F= V\cap E=\varnothing$ and $U$ is homeomorphic to $V$. Then $\dim X = \dim Y$. 
\end{theorem}

\begin{proof}
By Theorem \ref{closed-subspace} we have $\rd_X U \leq \dim X$ and by Lemma \ref{fin-points} $\dim X \leq \rd_X U$, so $\dim X=\rd_X U$. By the same argument $\dim Y = \rd_Y V$. Since $X, Y$ are compact we have 
$$\rd_X U= \sup \{ \dim F: F\subseteq U, F \ \text{compact}\} $$
and 
$$\rd_Y V= \sup \{ \dim F: F\subseteq V, F \ \text{compact}\}. $$
But $U$ and $V$ are homeomorphic, so every compact subset of $U$ is homeomorphic to some compact subset of $V$ and vice versa, and hence $\rd_X U=\rd_Y V$. This gives $\dim X= \rd_X U= \rd_Y V= \dim Y$. 
\end{proof}

\begin{theorem}\label{measure-zero-dim}
Suppose that $K$ is a metrizable compact space and $\mu$ is a non-zero Radon measure on $K$. Then there is a compact zero-dimensional subset $Z\subseteq K$ such that $\mu(Z)\neq 0$. 
\end{theorem}

\begin{proof}
Let $\{d_n \}_{n\in\omega}$ be a countable dense subset of $K$. For every $n\in\omega $ pick a countable local base $\{U_i^n\}_{i\in\omega}$ at $d_n$ such that $\overline{U}_{i+1} \subseteq U_i$ for $i\in\omega$. Then for every $n\in \omega$ there is $k_n\in\omega$ such that $$\sum_{i=k_n}^\infty |\mu|(\partial U_i^n)<\frac{\|\mu\|}{2^{n+1}}.$$
In particular we have $$|\mu|(Y)\neq 0$$
where $$Y=K\backslash \bigcup_{n=0}^\infty\bigcup_{i= k_n}^\infty \partial U_i^n.$$
Moreover, $Y$ is zero-dimensional, since $\{U_i^n\cap Y: n\in\omega, i\geq k_n\}= \{(U_i^n\backslash \partial U_i^n)\cap Y: n\in\omega, i\geq k_n\}$ forms a base of $Y$ consisting of clopen sets. By regularity of $\mu$ there is a compact set $Z\subseteq Y$ with $\mu(Z)\neq 0$ which is zero-dimensional as a compact subset of zero-dimensional space $Y$.
\end{proof}

\section{Spaces $C(K)$ with few operators}\label{section-few-operators}
We will follow the terminology from \cite{koszmider-shelah}.
We say that a bounded linear operator $T:C(K)\rightarrow C(K)$ is a weak multiplication, if it is of the form $T=gI+S$, where $g$ is a continuous function on $K$, $I$ is the identity operator and $S:C(K)\rightarrow C(K)$ is weakly compact. $T$ is called a weak multiplier, if $T^*=gI+S$ for some bounded Borel map $g:K\rightarrow \R$ and weakly compact $S:C(K)^*\rightarrow C(K)^*$. 

\begin{definition}\label{few-operators-def}
Let $K$ be a compact Hausdorff space. We say that the Banach space $C(K)$ has few operators if every bounded linear operator $T:C(K) \rightarrow C(K)$ is a weak multiplication.
\end{definition}

\begin{lemma}\label{ccc}
Suppose that $K$ is a c.c.c. compact Hausdorff space and that $C(K)\sim C(L)$ for a compact Hausdorff space $L$. Then $L$ is also c.c.c.
\end{lemma}

\begin{proof}
By \cite[Theorem 4.5(a)]{rosenthal} a compact space $M$ is c.c.c. if and only if $C(M)$ contains no isomorphic copy of $c_0(\omega_1)$, so in particular given property is an isomorphism invariant.
\end{proof}

\begin{lemma}\label{convergent-sequence}
Let $K$ be a compact Hausdorff space. If $K$ has a non-trivial convergent sequence, then $C(K)$ admits a complemented copy of $c_0$. In particular, if $C(K)$ has few operators, then $K$ has no non-trivial convergent sequences. 
\end{lemma}

\begin{proof}
The fact that non-trivial convergent sequences give rise to complemented copies of $c_0$ is well-known (cf. \cite{kania}). The second part of the lemma follows from \cite[Theorem 13 (3)]{few-operators-survey}.
\end{proof}

\begin{lemma}\label{isolated-points}
Assume that $K$ is a separable connected compact Hausdorff space such that $C(K)$ has few operators and $L$ is a compact Hausdorff space such that $C(K)\sim C(L)$. Let $J$ be the set of isolated points in $L$ and $L'=L\backslash J$. Then $J$ is a countable set and $L'$ has no isolated points. 
\end{lemma}

\begin{proof}
Since $K$ is separable, it is c.c.c., so by Lemma \ref{ccc} $L$ is also c.c.c. In particular $J$ is countable. 

Obviously, if $J$ is finite, then $L'$ has no isolated points, so we may assume that $J$ is infinite. 
Suppose that $x\in L'$ is an isolated point. Then $L'\backslash \{x\}$ is a closed subspace of $L$, so there is an open set $V\subseteq L$ such that $x\in V$ and $\overline{V}\cap (L'\backslash \{x\})=\varnothing$. $\overline{V}\subseteq J\cup \{x\}$, so $\overline{V}$ is an infinite countable compact space with exactly one non-isolated point i.e. it is a convergent sequence. By Lemma \ref{convergent-sequence} $C(L)$ admits a complemented copy of $c_0$, and so $C(K)$ admits a complemented copy of $c_0$. However, it is impossible since by \cite[Theorem 13 (a)]{few-operators-survey} $C(K)$ is indecomposable. 
\end{proof}

\begin{definition}
For a compact space $K$ and a function $f\in C(K)$ we denote by $M_f$ the operator $M_f:C(K)\rightarrow C(K)$ given by $M_f(g)=fg$.  
\end{definition}

In the next lemmas we will use the following characterization of weakly compact operators on Banach spaces of continuous functions from \cite[p. 160]{diestel-uhl}.

\begin{theorem}\label{weakly-compact-thm}
If $K$ is a compact Hausdorff space, then an operator $T$ on $C(K)$ is weakly compact if and only if for every bounded sequence $(e_n)_{n\in \omega}$ of pairwise disjoint functions (i.e. $e_n\cdot e_m =0$ for $n\neq m$) we have $\lim_{n\rightarrow \infty} \|T(e_n)\|=0$.
\end{theorem}

\begin{lemma}\label{weakly-compact-mult}
Let $L$ be a compact Hausdorff space, $J$ the set of isolated points in $L$, and $L'=L\backslash J$. Assume that $f\in C(L)$ is such that $f|L'=0$. Then $M_f$ is weakly compact.
\end{lemma}

\begin{proof}
Fix any bounded pairwise disjoint sequence $(e_n)_{n\in \omega}$ of elements of $C(L)$. Without loss of generality we may assume that $\|e_n\|\leq 1$ for each $n$. Let $\varepsilon>0$. 
Since $f$ is continuous and equal to $0$ on $L'$ there is only finitely many points $x$ such that $|f(x)|\geq \varepsilon$. Hence for $n$ large enough we have $\|M_f(e_n)\|=\|fe_n\|<\varepsilon$, which means that $\lim_{n\rightarrow \infty} \|M_f(e_n)\|=0$. Now Theorem \ref{weakly-compact-thm} says that $M_f$ is weakly compact. 
\end{proof}

\begin{lemma}\label{weakly-compact-mult2}
Assume that $K$ has no isolated points and $f\in C(K)$ is such that $M_f$ is weakly compact. Then $f=0$. 
\end{lemma}

\begin{proof}
Assume that $f\neq 0$. Then there is non-empty open set $U\subset K$ such that $|f(x)|\geq \varepsilon$ for $x\in U$ and some $\varepsilon>0$. Since there are no isolated points in $K$, $U$ is infinite, so there are pairwise disjoint open subsets $U_n\subseteq U$. Let $e_n\in C(K)$ be such that $e_n(x)=1$ for some $x\in U_n, e_n(x)=0$ for $x\in K\backslash U_n$ and $\|e_n\|=1$. Then for each $n\in\omega$ we have $\|M_fe_n\|\geq \varepsilon$, so by Theorem \ref{weakly-compact-thm} $M_f$ is not weakly compact.   
\end{proof}

\begin{lemma}\label{norm-mult} 
Let $f\in C(L)$ for $L$ compact Hausdorff and assume that there is a non-isolated point $x_0\in L$ such that $|f(x_0)|=\|f\|$. If $R:C(L)\rightarrow C(L)$ is a weakly compact operator, then $\|f\|\leq\|M_f+R\|$. 
\end{lemma}

\begin{proof}
Since $x_0$ is non-isolated there are distinct points $x_n\in L$ such that $|f(x_n)|>\|f\|-1/n$. By passing to a subsequence we may assume that $\{x_n: n\in\omega\}$ is a relatively discrete subset of $L$.  

Take pairwise disjoint open sets $U_n\subseteq  \{x\in K: |f(x)|>\|f\|-1/n\}, x_n\in U_n$. For each $n\in\omega$ pick $e_n\in C(L)$ such that $\|e_n\|=1$ and $e_n|(L\backslash U_n)=0$. In particular $(e_n)_{n\in\omega}$ are pairwise disjoint functions, so by Theorem \ref{weakly-compact-thm} $\lim_{n\rightarrow \infty} \|R(e_n)\|=0$. Moreover, $\|M_f(e_n)\|= \|fe_n\|\geq \|f\|-1/n$ (from the property of $U_n$). Hence we get that $\|M_f+R\|\geq \|(M_f+R)(e_n)\|= \|M_f(e_n)+R(e_n)\|\geq \|M_f(e_n)\|-\|R(e_n)\|\geq \|f\|-1/n-\|R(e_n)\|$. By taking limit with $n\rightarrow \infty$ we get $\|M_f+R\|\geq \|f\|$. 
\end{proof}

\begin{remark}\label{algebra-isomorphism}
If $K$ and $L$ are compact Hausdorff spaces, and $T:C(K)\rightarrow C(L)$ is an isomorphism of Banach spaces, then $T$ induces an isomorphism of the Banach algebras $\Phi_T:\mathcal{B}(C(L))\rightarrow \mathcal{B}(C(K))$ given by $$\Phi_T(U)=T^{-1}UT.$$

If $R\in \mathcal{B}(C(L))$ is a weakly compact operator, then $\Phi_T(R)$ is also weakly compact as a composition of a weakly compact operator with bounded operators. Similarly, if $S\in \mathcal{B}(C(K))$ is weakly compact, then $\Phi_T^{-1}(S)$ is weakly compact. 
\end{remark}

For the rest of this section we will assume that $K$ and $L$ are compact Hausdorff spaces, $L'$ is the set of non-isolated points of $L$, $C(K)$ has few operators and $T:C(K)\rightarrow C(L)$ is an isomorphism of Banach spaces.

\begin{definition}\label{induced-operator}
Let $\Phi_T$ be such as in Remark \ref{algebra-isomorphism}. We define an operator $\Psi_T:C(L')\rightarrow C(K)$ by putting for each $f'\in C(L')$ $$\Psi_T(f')=g,$$ for $g\in C(K)$ satisfying $\Phi_T(M_f)=M_g+R,$ where $R$ is weakly compact and $f\in C(L)$ is such that $f'=f|L'$.

In other words, $\Psi_T$ is defined in the way such that the following diagram commutes:
\begin{center}
\begin{tikzcd}
C(L) \arrow[d, "\mathcal{R}"'] \arrow[r, "M"] &  \mathcal{B}(C(L)) \arrow[r, "\Phi_T"] & \mathcal{B} (C(K)) \arrow[r, "\pi"] \arrow[l] & \mathcal{B} (C(K))/\mathcal{WC}(C(K)) \arrow[d, "\mathcal{I}"] \\
C(L')  \arrow[rrr, "\Psi_T"', dashed]         &                                        &                                               & C(K) \arrow[u]           
\end{tikzcd}
\end{center}
Here $\mathcal{R}$ stands for the restriction operator (i.e. $\mathcal{R}(f)=f|L'$), $M(f)=M_f$, $\pi$ is the natural surjection onto the quotient algebra $\mathcal{B} (C(K))/\mathcal{WC}(C(K))$, where $\mathcal{WC}(C(K))$ is the closed ideal in $\mathcal{B} (C(K))$ consisting of weakly compact operators and $\mathcal{I}:\mathcal{B} (C(K))/\mathcal{WC}(C(K))\rightarrow C(K)$ is the isometry given by $\mathcal{I}([M_g])=g$.     
\end{definition}

\begin{lemma}\label{well-defined}
Suppose that $K$ has no isolated points. Then the induced operator $\Psi_T:C(L')\rightarrow C(K)$ from Definition \ref{induced-operator} is a well-defined bounded linear and multiplicative operator. 
\end{lemma}

\begin{proof}
Take any $f'\in C(L')$ and
let $f_1,f_2\in C(L)$ and $g_1, g_2\in C(K)$ be such that $f_1|L'=f_2|L'=f'$ and $$\Phi_T(M_{f_i})=M_{g_i}+R_i \ \text{for} \ i=1,2,$$
where $R_1,R_2$ are weakly compact. 
Then $(f_1-f_2)|L'=0$, so by Lemma \ref{weakly-compact-mult} $M_{f_1}-M_{f_2}=M_{f_1-f_2}$ is weakly compact. This implies that 
\begin{gather*}
M_{g_1-g_2}=M_{g_1}-M_{g_2}= R_1-\Phi_T(M_{f_1})-R_2+\Phi_T(M_{f_2}) = \\ = R_1-R_2-\Phi_T(M_{f_1}-M_{f_2})
\end{gather*}
is weakly compact since $\Phi_T(M_{f_1}-M_{f_2})$ is weakly compact (cf. Remark \ref{algebra-isomorphism}). Since $K$ has no isolated points, Lemma \ref{weakly-compact-mult2} implies that $g_1-g_2=0$, so $\Psi_T$ is well-defined.

For the linearity and multiplicativeness fix $f'_1=f_1|L', f'_2=f_2|L'\in C(L), a,b\in \R$ and put $\Psi_T(f'_1)=g_1, \Psi_T(f'_2)=g_2$. We have
\begin{gather*}
\Phi_T(M_{af_1+bf_2})= \Phi_T(aM_{f_1}+bM_{f_2})= a\Phi_T(M_{f_1})+b\Phi_T(M_{f_2})= \\ = M_{ag_1}+aR_1+M_{bg_2}+bR_2 = M_{ag_1+bg_2}+aR_1+bR_2
\end{gather*}
and
\begin{gather*}
\Phi_T(M_{f_1f_2})= \Phi_T(M_{f_1}M_{f_2})= \Phi_T(M_{f_1})\Phi_T(M_{f_2})= \\ = (M_{g_1}+R_1)(M_{g_2}+R_2)= M_{g_1g_2}+R_1M_{g_2}+M_{g_1}R_2+R_1R_2.
\end{gather*}
But $aR_1+bR_2$ and $R_1M_{g_2}+M_{g_1}R_2+R_1R_2$ are weakly compact as the sums of weakly compact operators composed with bounded operators. Hence $\Psi_T(af'_1+bf'_2)=ag_1+bg_2$ and $\Psi_T(f'_1f'_2)=g_1g_2$.

Now we will show that $\Psi_T$ is bounded. Pick any $f'\in C(L')$. By the Tietze theorem $f'$ has an extension $f\in C(L)$ satisfying $\|f\|=\|f'\|$.  From Lemma \ref{norm-mult} we get that if $\Phi_T(M_f)=M_g+R$, then $\|g\|\leq \|M_g+R\|\leq \|\Phi_T\|\|M_f\|=\|\Phi_T\|\|f\|= \|\Phi_T\|\|f'\|$, so $\|\Psi_T\|\leq \|\Phi_T\|$.
\end{proof}

\begin{lemma}\label{injective}
Suppose that $K$ is separable and connected.
Then there is $c>0$ such that for every $f'\in C(L')$ we have $\|\Psi_T(f')\|\geq c\|f'\|$ i.e. $\Psi_T$ is an isomorphism onto its range. In particular $\Psi_T$ has closed range. 
\end{lemma}

\begin{proof}
Assume that $\Psi_T(f')=g$. Let $f\in C(L)$ be an extension of $f'$ such that $\|f\|=\|f'\|$. We have $\Phi_T(M_f)=M_g+R$ for some weakly compact $R$, so $\Phi_T^{-1}(M_g)=M_f-\Phi_T^{-1}(R)$. $\Phi_T^{-1}(R)$ is weakly compact by Remark \ref{algebra-isomorphism}, so from Lemma \ref{norm-mult} we get 
\begin{gather*}
\|f\|\leq\|M_f-\Phi_T^{-1}(R)\|=\|\Phi_T^{-1}\circ\Phi_T(M_f-\Phi_T^{-1}(R)) \|=\|\Phi_T^{-1}(M_g+R-R) \|=\\=\|\Phi_T^{-1}(M_g)\|\leq \|\Phi_T^{-1}\|\|M_g\|=\|\Phi_T^{-1}\|\|g\|.
\end{gather*}
Hence it is enough to take $c=\frac{1}{\|\Phi_T^{-1}\|}$.
\end{proof}

\begin{proposition}\label{S-separable}
Suppose that $K$ is separable and connected.\\
Let $S:C(K)\rightarrow C(K)$ be given by $S(f)= \Psi_T(T(f)|L')$. Then $$\ker(S)=T^{-1}(\{g\in C(L): g|L'=0\})$$ and it is a separable subspace of $C(K)$. 

\begin{center}
\begin{tikzcd}
C(K) \arrow[r, "T"] \arrow[rrr, "S"', bend right] & C(L) \arrow[r, "\mathcal R"] & C(L') \arrow[r, "\Psi_T"] & C(K)
\end{tikzcd}
\end{center}
\end{proposition}

\begin{proof}
By Lemma \ref{isolated-points} the set $J$ of isolated points in $L$ is countable, so we may write $J=\{x_n:n\in \omega\}$. Let $\chi_{\{x_n\}}$ be the characteristic function of $\{x_n\}$.  Observe that $\overline{\lin \{\chi_{\{x_n\}} :n\in \omega\} }= \{g\in C(L): g|L'=0\}$ is a separable subspace of $C(L)$, so it is enough to show that $\ker(S)= T^{-1}(\{g\in C(L): g|L'=0\})$, since $T$ is an isomorphism.

Assume that $S(f)=0$. Then $\Psi_T(T(f)|L')=0$, so $\Phi_T(M_{T(f)})=M_0+R=R$ is weakly compact and hence $M_{T(f)}=T\Phi_T(M_{T(f)})T^{-1}$ is also weakly compact as a composition of a weakly compact operator with bounded operators. From Theorem \ref{weakly-compact-thm} $\lim_{n\rightarrow \infty}\|T(f)e_n\| = 0$ for every bounded disjoint sequence $(e_n)_{n\in \omega}$. This implies that $\lim_{n\rightarrow \infty}\|(T(f)|L')e_n\| = 0$ for every bounded disjoint sequence $(e_n)_{n\in \omega}$. By applying Theorem \ref{weakly-compact-thm} once again we get that $M_{T(f)|L'}$ is weakly compact as an operator on $C(L')$. Since $L'$ has no isolated points (cf. Lemma \ref{isolated-points}) we get that $T(f)|L'=0$ by Lemma \ref{weakly-compact-mult2} i.e. $f\in T^{-1}(\{g\in C(L): g|L'=0\})$, so  $\ker(S)\subseteq T^{-1}(\{g\in C(L): g|L'=0\})$.

If $g\in C(L)$ is such that $g|L'=0$, then by Lemma \ref{weakly-compact-mult} $M_g$ is weakly compact, so $S(T^{-1}(g))=\Psi(g|L')=\Psi(0)=0$ and hence $T^{-1}(g)\in\ker(S)$.
\end{proof}

\begin{proposition}\label{Me-iso}
Suppose that $K$ is separable and connected. Let $S= \Psi_T(T(f)|L')$.
Write $S$ as a sum $S=M_e+W$ with $W$ weakly compact. Then $M_e$ is an isomorphism of $C(K)$.
\end{proposition}

\begin{proof}
It is enough to prove that $e(x)\neq 0$ for every $x\in K$. Indeed, if it is the case, then $M_g$ is the inverse of $M_e$ for $g=\frac{1}{e}$.

Assume that $e(z)=0$ for some $z\in K$ and aim for a contradiction. Then using the technique from the proof of Lemma \ref{norm-mult} we construct pairwise disjoint non-empty open subsets $U_n\subseteq K$ such that $\|e|U_n\|\leq \frac{1}{n}$ for each $n\in\omega$. Let $V_n$ be non-empty open sets such that $\overline{V}_n\subseteq U_n$. 

By Lemma \ref{convergent-sequence} $K$ has no convergent sequences and hence for every $n\in\omega$ the space $\overline{V}_n$ is non-metrizable as an infinite (because $V_n$ has no isolated points) compact set without convergent sequences.  We get that points in $\overline{V}_n$ cannot be separated by countable family of continuous functions (otherwise, if $(f_n)_{n\in\omega}$ separated points of $\overline{V}_n$, $(f_1,f_2,\dots):\overline{V}_n \rightarrow \R^n$ would be a homeomorphism onto a compact subspace of metrizable space), so since $\ker(S)$ is separable, there are points $x_n, y_n\in \overline{V}_n\subseteq U_n$ such that $d(x_n)=d(y_n)$ for all $d\in \ker(S)$. Let $f_n\in C(K)$ be such that $\|f_n\|=1, f_n(x_n)=1,f_n(y_n)=0$ and $f_n|(K\backslash U_n)=0$. Then for all $d\in \ker(S)$
\begin{gather*}
\|f_n-d\|\geq \max \{|f_n(x_n)-d(x_n)|, |f_n(y_n)-d(y_n)|\}=\\= \max \{|1-d(x_n)|, |d(x_n)|\}\geq 1/2.
\end{gather*}  
Since $f_n|(K\backslash U_n)=0$ and $\|e|U_n\|\leq \frac{1}{n}$ we have $\|ef_n\|\leq \frac{1}{n}$, so $\lim_{n\rightarrow\infty} \|ef_n\|=0$. 

$\Psi_T$ has closed range (Lemma \ref{injective}) and $T, \mathcal{R}$ are surjective, so $S$ has also closed range. By the first isomorphism theorem (see e.g. \cite[Corollary 2.26]{banach-space-theory}) $S[C(K)]$ is isomorphic to $C(K)/\ker (S)$, so since the distance of $f_n$ from $\ker (S)$ is greater than $1/2$ for all $n\in\omega$, there is $c>0$ such that  $\|S(f_n)\|>c$ for all $n\in\omega$.  

\begin{center}
\begin{tikzcd}
C(K) \arrow[rr, "S"] \arrow[rd] &                        & {S[C(K)]} \arrow[ld, "\sim"] \\
                                & C(K)/\ker (S) \arrow[ru] &                             
\end{tikzcd}
\end{center}
But on the other hand we have 
$$\|S(f_n)\|=\|ef_n+W(f_n)\|\leq \|ef_n\|+\|W(f_n)\|\rightarrow 0$$ when $n\rightarrow \infty$, since we have $\lim_{n\rightarrow\infty} \|ef_n\|=0$ and $\lim_{n\rightarrow\infty} \|W(f_n)\|=0$ (because $W$ is weakly compact and $(f_n)$ are bounded and pairwise disjoint), so we get a contradiction.
\end{proof}

Recall that an operator $R:X\rightarrow Y$ is called strictly singular, if for every infinite-dimensional subspace $X'\subseteq X$ the restriction $R|X'$ is not isomorphism. We cite the result from \cite{strictly-singular}.

\begin{theorem}\label{strictly-singular}
Let $X$ be a compact Hausdorff space. A bounded operator $R:C(X)\rightarrow C(X)$ is weakly compact if and only if it is strictly singular.
\end{theorem}

If we apply the above theorem to \cite[Proposition 2.c.10]{LT} we get the following.

\begin{theorem}\label{fredholm}
Let $E:C(X)\rightarrow C(X)$ be an operator with closed range for which $\dim \ker(E)<\infty$ and $\dim (C(X)/E(C(X))) <\infty$. Let $R: C(X)\rightarrow C(X)$ be weakly compact. Then $E+R$ also has closed range and $\dim ((C(X))/(E+R)(C(X))<\infty$.
\end{theorem}

\begin{corollary}\label{cofinite-dimensional}
Suppose that $K$ is separable and connected. Let $S= \Psi_T(T(f)|L')$.

Then the range of $S$ is finite-codimensional in $C(K)$. In particular the range of $\Psi_T$ is finite-codimensional in $C(K)$. 
\end{corollary}

\begin{proof}
Since $M_e$ is an isomorphism (by Proposition \ref{Me-iso}) and $W$ is weakly compact we may apply Theorem \ref{fredholm} to $S=M_e+W$.
\end{proof}

Since $\Psi_T:C(L')\rightarrow C(K)$ is a bounded linear multiplicative operator (Lemma \ref{well-defined}), there is $\varphi:K\rightarrow L'$ such that $\Psi_T(f)=f\circ \varphi$ for $f\in C(L')$ (see e.g. \cite[Theorem 7.7.1]{semadeni}).
From Lemma \ref{injective} and Corollary \ref{cofinite-dimensional} we get that $\Psi_T$ is an embedding with finite-codimensional range, so the induced map $\varphi$ is surjective and has only finitely many fibers containing more than one element and each of these fibers is finite. In particular $K=U\cup F$ where $F$ is a finite set and $\varphi|U$ is a homeomorphism and we get the following theorem.

\begin{theorem}\label{few-operators-modulo-finite}
Suppose that $K$ is a separable connected compact Hausdorff space such that $C(K)$ has few operators and $L$ is a compact Hausdorff space such that $C(K)\sim C(L)$. 

Then $K$ and $L$ are homeomorphic modulo finite set i.e. there are open subsets $U\subseteq K, V\subseteq L$ and finite sets $E\subseteq K, F\subseteq L$ such that $U, V$ are homeomorphic and $K=U\cup E, L=V\cup F$.
\end{theorem}



\begin{corollary}\label{corollary-dimension-equal}
If $\dim(K)=n$ and $K$ is a compact, separable and connected Hausdorff space such that $C(K)$ has few operators, then for each compact Hausdorff space $L$ such that $C(K)\sim C(L)$ we have $\dim(L)=n$.
\end{corollary}
\begin{proof}
Use Theorem \ref{few-operators-modulo-finite} and Theorem \ref{dim-equal}.
\end{proof}

\section{Extensions of compact spaces}\label{section-extensions}
In this section we consider the notion of strong extension from \cite{few-operators-main}. We describe the methods of controlling the dimension in constructions of compact spaces using strong extensions. We prove that strong extensions cannot lower the dimension of initial space and we show how to construct extensions that cannot rise the dimension. 
\begin{definition}
Let $K$ be a compact Hausdorff space and $(f_n)_{n\in\omega}$ be a sequence of pairwise disjoint continuous functions $f_n:K\rightarrow [0,1]$. Define 
$$D((f_n)_{n\in\omega}) = \bigcup \{U: U \ \text{is open and} \ \{n : \supp (f_n) \cap U \neq \varnothing \} \ \text{is finite} \}. $$
We say that $L\subseteq K\times [0,1]$ is the extension of $K$ by $(f_n)_{n\in\omega}$ if and only if $L$ is the closure of the graph of $(\sum_{n\in \omega} f_n)|D((f_n)_{n\in\omega})$. We say that this is a strong extension, if the graph of $\sum_{n\in \omega} f_n$ is a subset of $L$.
\end{definition}

\begin{lemma}\cite[Lemma 4.1]{few-operators-main}\label{continuous-sum}
If $(f_n)_{n\in\omega}$ are pairwise disjoint continuous functions on $K$ with values in $[0,1]$, then $\sum_{n\in\omega} f_n$ is well-defined and continuous in the dense open set $D((f_n)_{n\in\omega})$.
\end{lemma}

\begin{lemma}\cite[Lemma 4.4]{few-operators-main}\label{connected-extension}
Strong extension of a connected compact Hausdorff space is connected. 
\end{lemma}

Note that there are known examples of extensions of connected compact spaces which are not connected (see \cite{barbeiro-fajardo}), so the assumption that considered extensions are strong is necessary. 

\begin{lemma}\label{extension-separable}
Let $K$ be a separable compact Hausdorff space with countable dense set $Q=\{q_n:\in \omega\}$ and let $L$ be an extension of $K$ with the natural projection $\pi:L\rightarrow K$. Assume that $Q'=\{q'_n: n\in\omega \}$ is a subset of $L$ such that $\pi(q'_n)=q_n$ for every $n\in \omega$. Then $Q'$ is a dense subset of $L$. 
\end{lemma}
\begin{proof}
Let $(f_n)_{n\in\omega}$ be a sequence of pairwise disjoint continuous functions such that $L$ is the extension of $K$ by $(f_n)_{n\in\omega}$. By \cite[Lemma 4.3 a)]{few-operators-main} $\pi^{-1}(D((f_n)_{n\in\omega})$ is dense in $L$. Moreover, $\pi|\pi^{-1}(D((f_n)_{n\in\omega}))$ is homeomorphism as a projection of graph of continuous function onto its domain. Since $Q$ is dense in $K$ and $D((f_n)_{n\in\omega})$ is open, $Q\cap D((f_n)_{n\in\omega})$ is dense in $D((f_n)_{n\in\omega})$. Hence we get that $\pi^{-1}(Q\cap D((f_n)_{n\in\omega}))$ is dense in $L$. But if $q_n\in  D((f_n)_{n\in\omega})$, then $\pi^{-1}(q_n)=\{q'_n\}$, so $Q'\supseteq \pi^{-1}(Q\cap D((f_n)_{n\in\omega})$ is also dense in $L$. 
\end{proof}

The following lemma is a special case of \cite[Lemma 4.5]{few-operators-main}.
\begin{lemma}\label{many-strong-extensions}
Suppose that $K$ is a compact metric space and that for every $n\in\omega$ $X^n_1,X^n_2$ are disjoint relatively discrete subsets of $K$ such that $\overline{X^n_1}\cap\overline{X^n_2}\neq \varnothing$. Let $(f_n)_{n\in\omega}$ be a pairwise disjoint sequence of continuous functions from $K$ into $[0,1]$. For an infinite subset $B\subseteq \omega$ denote by $K(B)$ the extension of $K$ by $(f_n)_{n\in B}$. For $i=0,1$ and $n\in \omega$ put 
$$X^n_i(B)=\{(x,t):x\in X^n_i, t=\sum_{k\in B} f_k(x)\}.$$
Then there is an infinite $N\subseteq \omega$ such that for every $B\subseteq N$:
\begin{enumerate}
    \item $K(B)$ is a strong extension of $K$ by $(f_n)_{n\in B}$,
    \item $\overline{X^n_1(B)}\cap \overline{X^n_2(B)}\neq \varnothing$ for every $n\in\omega$, where the closures are taken in $K(B)$. 
\end{enumerate}
\end{lemma}

\begin{proposition}\label{extension-property-E}
If $L$ is a strong extension of a compact Hausdorff space $K$ with the natural projection $\pi:L\rightarrow K$, then $\pi$ is essential-preserving. 
\end{proposition}

\begin{proof}
Let $(f_k)_{k\in \omega}$ be such that $L$ is a strong extension of $K$ by $(f_k)_{k\in \omega}$. \\ Let $\{(A_i,B_i): i=1,2,\dots,n\}$ be an essential family in $K$ and assume that $\{(\pi^{-1}(A_i), \pi^{-1}(B_i)):i=1,2,\dots,n\}$ is not essential in $L$. By Theorem \ref{essential-char} there are closed sets $C_i\supseteq \pi^{-1}(A_i), D_i\supseteq \pi^{-1}(B_i)$ such that $C_i\cap D_i =\varnothing$ for each $i\leq n$ and
$$\bigcup_{i=1}^n (C_i\cup D_i)=L.$$
Since $C_i, D_i$ are compact, there are sets $U_i,V_i$ open in $K\times[0,1]$ such that $C_i\subseteq U_i$, $D_i\subseteq V_i$ and $U_i\cap V_i=\varnothing$ for every $i\leq n$. 

For each $k\in \omega$ denote by $L_k$ the graph of $\sum_{i\leq k} f_i$ and let $\pi_k:L_k\rightarrow K$ be the projection onto $K$. 
\begin{claim}
For every $k\in\omega$ we have $$L_k\backslash \bigcup_{i=1}^n(U_i\cup V_i)\neq \varnothing.$$
\end{claim}

\begin{proof2} Assume that there is $N$ such that $$L_N\subseteq \bigcup_{i=1}^n(U_i\cup V_i).$$ Then for every $k\geq N$ 
\begin{gather}
    L_k\backslash L_N= \graph({\sum_{i=N+1}^k f_i| \supp(\sum_{i=N+1}^k f_i)})\subseteq L\subseteq \bigcup_{i=1}^n(U_i\cup V_i)
\end{gather} 
(the first equality holds, because the supports of $f_i$'s are pairwise disjoint), so we have 
$$L_k\subseteq \bigcup_{i=1}^n(U_i\cup V_i).$$
Put $A_i^k= \pi_k^{-1}(A_i), B_i^k= \pi_k^{-1}(B_i)$ and observe that the family $\{(A_i^k,B_i^k): i\leq n\}$ is essential in $L_k$ since $\pi_k$ is a homeomorphism. Hence there is $i\leq n$ such that $A_i^k\nsubseteq U_i$ or $B_i^k\nsubseteq V_i$. Indeed, otherwise $U_i\cap L_k, V_i\cap L_k$ would be disjoint open subsets of $L_k$ with 
$$\bigcup_{i=1}^n((U_i\cap L_k)\cup (V_i\cap L_k))=L_k,$$
which contradicts the fact that $\{(A_i^k,B_i^k): i\leq n\}$ is essential (cf. Theorem \ref{essential-char}). 
Without loss of generality there are infinitely many $k$ such that $A_1^k\backslash U_1\neq \varnothing$. 
For every $k\in \omega$ we have $$A_1^{k+1}\backslash A_1^k = \pi_{k+1}^{-1}(A_1)\backslash \pi_k^{-1}(A_1)=\graph(f_{k+1}|(A_1\cap \supp(f_{k+1}))\subseteq \pi^{-1}(A_1)\subseteq U_1.$$
In particular $(A_1^k\backslash U_1)_{k\in\omega}$ form a decreasing sequence of non-empty compact sets. Hence $$A=\bigcap_{k=1}^\infty A_1^k\backslash U_1\neq \varnothing.$$ We have $A\subseteq L$ since if $(x,t)\in A$, then $f_k(x)=0$ for all $k$, so $\sum_{k\in\omega} f_k(x)=0$ and hence $(x,t)=(x,0)$ is an element of the graph of $\sum_{k\in\omega} f_k$ which is a subset of $L$. Moreover $A\subseteq A_1\times[0,1]$, so $A\subseteq (A_1\times[0,1])\cap L=\pi^{-1}(A_1)$ which contradicts the assumption that $\pi^{-1}(A_1)\subseteq U_1$ and completes the proof of the claim. 
\end{proof2}
To finish the proof of the proposition put $$F_k=L_k\backslash \bigcup_{i=1}^n(U_i\cup V_i)$$ and observe that $(F_k)_{k\in\omega}$ is a decreasing sequence of non-empty compact sets (by (1) from the claim), so as in the case of the set $A$ from the claim we get that
$$F=\bigcap_{k=1}^\infty F_k$$ is a non-empty subset of the graph of  $\sum_{k\in\omega} f_k$, so $F\subseteq L$ (because the extension is strong), which is a contradiction, since $F$ is disjoint from $\bigcup_{i\leq n} (U_i\cup V_i)\supseteq L$. 
\end{proof}

\begin{lemma}\label{extension-dimension}
Suppose that $K$ is a compact metric space with $0<\dim(K)\leq n$ and $f_k:K\rightarrow [0,1]$ are pairwise disjoint continuous functions such that the set $$ Z=K\backslash D((f_k)_{k\in\omega}) $$ 
is zero-dimensional. Assume that $L$ is a strong extension of $K$ by $(f_k)_{k\in\omega}$. Then $\dim L\leq n$.
\end{lemma}
\begin{proof}
Let $\pi$ be the natural projection from $L$ onto $K$.
$\pi^{-1}(D((f_k)_{k\in\omega}))$ is an open subset of a metric space, so it is a union of countably many closed sets, each of dimension at most $n$ since $\pi^{-1}(D((f_k)_{k\in\omega}))$ is homeomorphic to $D((f_k)_{k\in\omega})$ (cf. Theorem \ref{closed-subspace}). The set $\pi^{-1}(Z)$ is included in $Z\times [0,1]$ so $\dim \pi^{-1}(Z)\leq 1\leq n$ by Theorem \ref{product-dim}. Hence $L=\pi^{-1}(D((f_k)_{k\in\omega}))\cup \pi^{-1}(Z)$ is a countable union of closed sets of dimension at most $n$. Now Theorem \ref{countable-sum-theorem} gives the inequality $\dim L\leq n$. 
\end{proof}

\begin{corollary}\label{corollary-dimension}
Let $\gamma$ be an ordinal number. Suppose that $\{K_\alpha:\alpha<\gamma \}$ is an inverse system of compact Hausdorff spaces such that: 
\begin{itemize}
    \item for every $\alpha$ the map $\pi^{\alpha+1}_\alpha: K_{\alpha+1}\rightarrow K_\alpha$ is a strong extension by pairwise disjoint continuous functions $(f_n^\alpha)_{n\in\omega}$ and the set $Z_\alpha= K_\alpha \backslash D((f_n^\alpha)_{n\in\omega})$ is zero-dimensional,
    \item if $\alpha$ is a limit ordinal, then $K_\alpha$ is the inverse limit of $\{K_\beta:\beta<\alpha\}$. 
\end{itemize}
Denote by $K_\gamma$ the inverse limit of $\{K_\alpha:\alpha<\gamma \}$. Then $\dim K_\gamma= \dim K_1$.
\end{corollary}
\begin{proof}
The inequality $\dim K_\gamma \geq \dim K_1$ follows from Proposition \ref{extension-property-E} and Theorem \ref{property-E}. The inequality $\dim K_\gamma \leq \dim K_1$ follows from Lemma \ref{extension-dimension} and Theorem \ref{inverse-dimension}.
\end{proof}

\section{The main construction}\label{section-construction}
\begin{theorem}\label{weak-multiplier}\cite[Lemma 2.4]{koszmider-shelah}
Suppose that $K$ is a compact Hausdorff space. If a bounded linear operator $T:C(K)\rightarrow C(K)$ is not a weak multiplier, then there are $\delta>0$, a pairwise disjoint sequence $(g_n)_{n\in\omega}\subseteq C_I(K)$ and pairwise disjoint open sets $(V_n)_{n\in\omega}$ such that 
$$\supp(g_n)\cap V_m=\varnothing$$ for all $n,m\in \omega$ and $$|T(g_n)|V_n|>\delta$$
for all $n\in\omega$. 
\end{theorem}
In particular, if $x_n\in V_n$ and $\mu_n=T^*(\delta_{x_n})$ for $n\in\omega$, then $|\int g_n d\mu_n|= |T(g_n)(q_{l_n})|>\delta$, and so $|\mu_n|(\supp (g_n))\geq |\int g_n d\mu_n|>\delta$. 

The idea behind the construction is as follows. We will construct a compact space $K$ as the inverse limit of spaces $K_\alpha \subseteq [0,1]^\alpha$ (so the final space is a subset of $[0,1]^{\mathfrak c}$). For each bounded sequence $(\mu_n)_{n\in\omega}$ of Radon measures on $[0,1]^{\mathfrak c}$ and a sequence of pairwise disjoint open sets $(V_n)_{n\in \omega}$ we want to use a strong extension in such a way that in the final space there will be no sequence $(g_n)_{n\in\omega}$ for which the properties from Theorem \ref{weak-multiplier} are satisfied. However, we need to consider $2^{\mathfrak c}$ sequences of Radon measures on $[0,1]^{\mathfrak c}$, while there are only $\mathfrak c$ steps in the construction. In order to handle this we will use $\diamondsuit$ (cf. Lemma \ref{diamond-lemma}).

\begin{proposition}\label{metric-functions}
Let $K$ be a compact metrizable space and $(\mu_n)_{n\in\omega}$ be a bounded sequence of Radon measures on $K$. Assume that $(U_n)_{n\in\omega}$ is a sequence of pairwise disjoint open sets and $\delta>0$ is such that $|\mu_n|(U_n)>\delta$ for $n\in\omega$. Then there is an infinite set $N\subseteq \omega$, continuous pairwise disjoint functions $f_n:K\rightarrow [0,1]$ and $\varepsilon>0$ such that 
\begin{enumerate}
    \item $\supp(f_n)\subseteq U_n$ for $n\in N$,
    \item $|\int f_n d\mu_n|>\varepsilon$ for $n\in N$,
    \item $\sum \{|\int f_m d\mu_n|: n\neq m, m\in N \}<\varepsilon/3$ for $n\in N$,
    \item $K\backslash D((f_n)_{n\in N})$ is zero-dimensional.
\end{enumerate}
\end{proposition}

\begin{proof}
Since $\mu_n$'s are Radon measures there is $\delta'>0$ and open sets $U'_n\subseteq U_n$ such that $|\mu_n(U'_n)|>\delta'$ for $n\in\omega$. Without loss of generality we may assume that $U'_n=U_n$ and $\delta'=\delta$. 

Put $\nu_n=\mu_n|U_n$ for $n\in\omega$. Let $N'$ be such that the sequence $(\nu_n)_{n\in N'}$ has the weak* limit $\nu$. Since $|\int 1 d\nu_n|>\delta $ for every $n$, we have  $|\int 1 d\nu|\geq\delta$, so $\nu$ is a non-zero measure. By Theorem \ref{measure-zero-dim} there is a compact zero-dimensional subset $Z\subseteq K$ and $\varepsilon>0$ such that $|\nu(Z)|>2\varepsilon$. Since $Z$ is a closed subset of a metrizable space and $\nu$ is a  regular measure, there is a decreasing sequence of open sets $(G_n)_{n\in N'}$ such that $Z= \bigcap G_n$ and $|\nu(G_n)|>2\varepsilon$ for all $n\in N'$. 

Note that if $f\in C_I(K)$ is such that $\supp(f)\subseteq G_n$ and $|\int f d\nu|>2\varepsilon$, then for big enough $l\in N'$ we have $|\int f d\nu_l|>2\varepsilon$ and so $|\nu_l|(G_n)=|\nu_l|(G_n\cap U_l)>2\varepsilon$. Hence for each $l\in N'$ we may pick $f_l\in C_I(K)$ such that $\supp f_l\subseteq G_n\cap U_l$ and $|\int f_l d\nu_l|=|\int f_l d\mu_l|>\varepsilon$. 

For each $n\in N'$ let $l_n\in N'$ be such that $\supp f_{l_n}\subseteq G_n\cap U_{l_n}$,  $|\int f_{l_n} d\mu_{l_n}|>\varepsilon$ and $(l_n)_{n\in N'}$ is an increasing sequence. Let $N''=\{l_n:n\in N'\}$. For every $M\subseteq N''$ denote $Z_M=K\backslash D((f_{l_n})_{n\in M})$. If $x\in K\backslash Z$, then there is an open neighbourhood $V\ni x$ such that for big enough $n\in M$ we have $V\cap G_n=\varnothing$ and so $V \cap \supp(f_{l_n})=\varnothing$. Hence $V\subseteq D((f_{l_n})_{n\in M})$, which gives $x\notin Z_M$. This implies that $Z_M\subseteq Z$, so in particular $Z_M$ is zero-dimensional and condition (4) is satisfied for any choice of $M\subseteq N''$. 
Now we use Rosenthal's lemma (see \cite[p. 82]{diestel} or \cite{sobota})  to obtain an infinite $N\subseteq N''$ such that the 3rd condition is also satisfied. 
\end{proof}

We will need the following strengthening of \cite[Lemma 6.2]{few-operators-main}. 
\begin{lemma}\label{lemma-for-odd}
Let $K$ be a compact, connected metrizable space with a countable dense set $Q=\{q_n: n\in \omega\}$. Let $U,V$ be open subsets of $K$ such that $\overline{U}\cap\overline{V}\neq \varnothing$. Then there is a sequence $(f_n)_{n\in\omega}$ of pairwise disjoint functions $f_n\in C_I(K)$ and infinite sets $A_0,A_1,S_0,S_1\subseteq \omega$ such that:
\begin{enumerate}
    \item the sets $\{q_n: n\in S_0\}\subseteq U, \{q_n: n\in S_1\}\subseteq V$ are relatively discrete,
    \item $A_i\subseteq S_i $ and $|S_i\backslash A_i| =\omega$ for $i=0,1$,
    \item for every infinite $B\subseteq \omega$ in the extension $K(B)$ of $K$ by $(f_n)_{n\in B}$ there are disjoint closed sets $F_0,F_1\subseteq K(B)$ and distinct $x_0,x_1\in K(B)$ such that for $i=0,1$ $$x_i\in \overline{\pi^{-1}(U)\cap \{q^B_n:n\in A_i\}}\cap \overline{\pi^{-1}(V)\cap \{q^B_n:n\in S_i\backslash A_i\}}$$
     and
$$\{q^B_n: n\in A_i\}\subseteq F_i,$$
where $q^B_j=(q_j,t)$ and $t=\sum_{n\in B} f_n(q_j)$,
    \item $|K\backslash D(f_n)_{n\in B}|=1$ (in particular $K\backslash D(f_n)_{n\in B}$ is zero-dimensional). 
\end{enumerate}   
\end{lemma} 

\begin{proof}
Fix any compatible metric $d$ on $K$.
Pick any $x\in \overline{U}\cap \overline{V}$. Since $K$ is connected, $x$ is not an isolated point. For $n\in\omega$ put $U'_n=U\cap \mathcal{B}(x,1/n), V'_n=V\cap \mathcal{B}(x,1/n)$ (where $\mathcal{B}(x,\varepsilon)$ is the open ball with $x$ as the center and radius $\varepsilon$ with respect to $d$) and let $U_n\subseteq U'_n, V_n\subseteq V'_n$ be non-empty open sets such that the members of the family $\{U_n,V_n:n\in\omega \}$ are pairwise disjoint.  Take continuous functions $f_n\in C(K)$ and $k_n, l_n\in \omega$ such that:
\begin{itemize}
    \item $q_{k_n}\in U_n, q_{l_n}\in V_n$,
    \item $f_n(q_{k_{2n}})=f_n(q_{l_{2n}})=1$,
    \item $\supp (f_n)\subseteq U_{2n}\cup V_{2n}$.
\end{itemize}
Let $B\subseteq \omega$ be infinite.  For (2) and (3) it is enough to take $S_0=\{{k_{2n+1}}, {l_{2n+1}}: n\in\omega \},  A_0= \{{k_{2n+1}}: n\in\omega \}, S_1=\{{k_{2n}}, {l_{2n}}: n\in\omega \}, A_1= \{{k_{2n}}: n\in\omega \}$, $x_0=(x,0), x_1=(x,1)$ and $F_0=K(B)\cap (K\times [0,1/3]), F_1=K(B)\cap (K\times [2/3,1])$. (1) is satisfied since $U_n,V_m$ are pairwise disjoint for $n,m\in \omega$.

(4) follows from the fact that $x$ is the only point all of whose neighborhoods intersect all but finitely many $U_n$'s and $V_n$'s, so we have $K\backslash D(f_n)_{n\in B}=\{x\}$.

\end{proof}

\begin{lemma}\label{diamond-lemma}
Assume $\diamondsuit$. Then there is a sequence $(M^\alpha, \mathcal{U}^\alpha, L^\alpha)_{\alpha<\omega_1}$ such that: 
\begin{itemize}
    \item $M^\alpha=(\mu^\alpha_n)_{n\in\omega}$ is a bounded sequence of Radon measures on $[0,1]^\alpha$,
    \item $\mathcal{U}^\alpha= (U_{n,m}^\alpha)_{n,m\in\omega}$ is a sequence of basic open sets in $[0,1]^\alpha$,
    \item $L^\alpha=(l_n^\alpha)_{n\in \omega} $ is a sequence of distinct natural numbers,
\end{itemize}
and for every:
\begin{itemize}
    \item bounded sequence $(\mu_n)_{n\in\omega}$ of Radon measures on $[0,1]^{\omega_1}$,
    \item  sequence $(U_{n,m})_{n,m\in\omega}$ of basic open sets in $[0,1]^{\omega_1}$,
    \item increasing sequence $l_n$ of natural numbers
\end{itemize}
there is a stationary set $S\subseteq \omega_1$ such that for $\beta\in S$ we have 
\begin{itemize}
    \item $\mu_n|C([0,1]^\beta)=\mu_n^\beta$,
    \item $\pi_\beta[U_{n,m}]=U_{n,m}^\beta$,
    \item $l_n=l_n^\beta$,
\end{itemize}
where $\pi_\beta$ denotes the natural projection from $[0,1]^{\omega_1}$ onto $[0,1]^\beta$.  
\end{lemma}

\begin{proof}
Firstly we will show that there is a sequence $(M_0^\alpha)_{\alpha<\omega_1}$ such that $M_0^\alpha=((\nu^\alpha_n)_{n\in\omega})$ is a bounded sequence of Radon measures on $[0,1]^\alpha$ and for every bounded sequence $(\nu_n)_{n\in\omega}$ of Radon measures on $[0,1]^{\omega_1}$ there is a stationary set $S\subseteq \omega_1$ such that for $\beta\in S$ we have $\nu_n|C([0,1]^\beta)=\nu_n^\beta$.

We will use the identification of Radon measures on $[0,1]^{\omega_1}$ with bounded functionals on $C([0,1]^{\omega_1})$ described in Section \ref{notation}. For a finite set $F\in [\omega_1]^{<\omega}$ denote by $w_F$ the product $\prod_{\alpha\in F} w_\alpha$, where $w_\alpha\in C([0,1]^{\omega_1}), w_\alpha(x)=x(\alpha)$. Observe that finite linear combinations of $w_F$'s form a subalgebra of $C([0,1]^{\omega_1})$. If $x,y \in [0,1]^{\omega_1}$ are distinct points with $x(\alpha)\neq y(\alpha)$, then $w_{\alpha}(x)\neq w_{\alpha}(y)$, so by the Stone-Weierstrass theorem this subalgebra is dense in $C([0,1]^{\omega_1})$. Hence if $\nu$ is a Radon measure on $[0,1]^{\omega_1}$ then it is determined by the values of $\nu(w_F)$ for $F\in [\omega_1]^{<\omega}$ (note also that in the same way if $\beta<\omega_1$, then $\nu|C([0,1]^\beta)$ is determined by the values of $\nu(w_F)$ for $F\in [\beta]^{<\omega}$). So we can represent each Radon measure $\nu$ on $[0,1]^{\omega_1}$ by the function $$\varphi_{\nu}:[\omega_1]^{<\omega}\rightarrow \R, \ \varphi(F)=\nu(w_F)$$ 
(and then $\nu|C([0,1]^\beta)$ is represented by $\varphi_\nu|[\beta]^{<\omega}$), and each countable sequence $M=(\nu_n)_{n\in\omega}$ we can represent by the function $$\varphi_M: [\omega_1]^{<\omega} \times \omega \rightarrow \R, \ \varphi_M(F, n)= \nu_n(w_F). $$

Let $\Phi_1: \omega_1 \rightarrow [\omega_1]^{<\omega}\times \omega $ be a bijection such that for each limit ordinal $\gamma\in Lim\cap \omega_1$ the restriction $\Phi_1|\gamma$ is bijection onto $[\gamma]^{<\omega}\times \omega$ (to see that such a bijection exists it is enough to note that for every $\gamma \in Lim\cap \omega_1$ there is a bijection $\phi_\gamma:[\gamma, \gamma+\omega)\rightarrow ([\gamma+\omega]^{<\omega}\times \omega)\backslash ([\gamma]^{<\omega}\times \omega)$ and take $\Phi_1|[\gamma, \gamma+\omega)=\phi_\gamma$). We need to fix one more bijection $\Phi_2 :\R \rightarrow \omega_1$ ($\diamondsuit$ implies {\sf CH}, so such a bijection exists). Put $$\psi_M=\Phi_2\circ \varphi_M\circ \Phi_1, \ \psi_M:\omega_1\rightarrow \omega_1.$$
Since $\Phi_1|\gamma$ is a bijection onto $[\gamma]^{<\omega}\times \omega$ for all limit $\gamma$ we may treat $\psi_M|\gamma$ as a representation of the sequence of measures $(\nu_n|C([0,1]^\gamma))_{n\in \omega}$. 

We will use the following characterization of $\diamondsuit$ (cf. \cite[Theorem 2.7]{devlin-diamond}): \\
There exists a sequence $(f_\alpha)_{\alpha<\omega_1}, f_\alpha: \alpha \rightarrow \alpha$ such that for for each $f:\omega_1\rightarrow \omega_1$ the set $\{\alpha: f|\alpha=f_\alpha \}$ is stationary.

For $\alpha\in \omega_1$ let $M_0^\alpha$ be a sequence of Radon measures on $[0,1]^\alpha$ represented by $f_\alpha$, if $f_\alpha$ is a representation for some such sequence (otherwise we pick $M_0^\alpha$ in any way).   
Let $M$ be a bounded sequence of of measures on $[0,1]^{\omega_1}$ and let $S=\{ \alpha: \psi_M|\alpha = f_\alpha \}$.
Since for limit $\gamma<\omega_1$ the function $\psi_M|\gamma$ is a representation of some sequence of measures we get that for $\alpha\in Lim\cap S$ the function $\psi_M|\alpha$ is the representation of a sequence $M_0^\alpha$. Moreover the set $S\cap Lim$ is a stationary subset of $\omega_1$, so the first part of the proof is complete.

To show the existence of sequence $(M^\alpha, \mathcal{U}^\alpha, L^\alpha)_{\alpha<\omega_1}$ required in the Lemma, we need to observe that each triple $(M, \mathcal{U}, L)$ may be represented as a bounded countable sequence of Radon measures on $[0,1]^{\omega_1}$. Indeed, any basic open set $U\in \mathcal{U}$ may be treated as a measure $\lambda_U$ on $[0,1]^{\omega_1}$, given by $\lambda_U(A)=\lambda(A\cap U)$, where $\lambda$ is a product measure of $\omega_1$ Lebesgue measures on $[0,1]$ (note that if $U,V$ are different basic open sets, then some of their sections differ on a non-trivial interval, so we have $\lambda_U\neq\lambda_V$) and $L$ may be represented as $\delta_{x_L}$ where $x_L=(y_l,0,0,\dots)$ and $y_l=g(L)$ for some fixed bijection $g$ between the set of sequences of natural numbers and $[0,1]$.
\end{proof}

\begin{proposition}\label{construction-properties}
Assume $\diamondsuit$. Then for every $k>0, k\in\omega\cup\{\infty\}$ there is a compact Hausdorff space $K$ satisfying the following properties: 
\begin{enumerate}
    \item $\dim K=k$,
    \item $K$ is separable with a countable dense set $Q=\{q_n:n\in\omega\}$,
    \item $K$ is connected,
    \item for every:
    \begin{itemize}
        \item sequence $(U_n)_{n\in\omega}$ of pairwise disjoint open sets which are countable unions of basic open sets (basic open set in $K$ is a set of the form $W\cap K$, where $W$ is a basic open set in $[0,1]^{\omega_1}$),
        \item relatively discrete sequence $(q_{l_n}: n\in\omega)\subseteq Q$ with $q_{l_n}\notin U_m$ for $n,m\in \omega$,
        \item bounded sequence $(\mu_n)_{n\in \omega}$ of Radon measures on $K$ such that $|\mu_n|(U_n)>\delta$ for some $\delta>0$,
    \end{itemize}
       there is $\varepsilon>0$, continuous functions $(f_n)_{n\in\omega}\subseteq C_I(K)$ and infinite sets $B\subseteq N\subseteq \omega$ such that:  
    \begin{enumerate}
    \item $(f_n)$ is a sequence of pairwise disjoint functions with $\supp(f_n)\subseteq U_n$ for $n\in\omega$,
    \item $|\int f_n d\mu_n|>\varepsilon$ for $n\in B$,
    \item $\sum \{|\int f_m d\mu_n|: m\in B\backslash\{n\} \}<\varepsilon/3$ for $n\in N$,
    \item $\{f_n:n\in B\}$ has supremum in the lattice $C(K)$,
    \item $\overline{\{q_{l_n}: n\in B\}}\cap \overline{\{q_{l_n}: n\in N\backslash B\}}\neq \varnothing,$
    \end{enumerate}
    \item whenever $U,V$ are open subsets of $K$ such that $\overline{U}\cap \overline{V}\neq \varnothing$, then $\overline{U}\cap \overline{V}$ contains at least two points.
\end{enumerate}
\end{proposition}
We will start with the description of the construction. Then we will prove that the constructed space satisfies the required conditions. 
\begin{construction}\label{construction}
Assume $\diamondsuit$. We will construct by induction on $\alpha<\omega_1$ an inverse system $(K_\alpha)_{\alpha<\omega_1}$ with the limit $K$, where $K_\alpha\subseteq [0,1]^\alpha$ and countable dense sets $Q_{\alpha}=\{q_n|\alpha:n\in\omega\}\subseteq K_\alpha$. 
\end{construction}
We start with $K_k=[0,1]^k$ (or $K_\omega=[0,1]^\omega$ in the case $k=\infty$) and we pick $Q_k$ to be any countable dense subset of $K_k$. If $\alpha$ is a limit ordinal then we take as $K_\alpha$ the inverse limit of $(K_\beta)_{\beta<\alpha}$. 

Denote by $\even$ and $\odd$ the sets consisting of even and odd (respectively) countable ordinals greater than $k$.  
Let $(M^\alpha, \mathcal{U}^\alpha, L^\alpha)_{\alpha<\omega_1}$ be as in Lemma \ref{diamond-lemma} and fix an enumeration $(U_\alpha, V_\alpha)_{\alpha\in \odd}$ of pairs of open subsets of $[0,1]^{\omega_1}$ which are countable unions of basic open sets, and require that each such a  pair occurs in the sequence uncountably many times (such an enumeration exists since by {\sf CH} there is $\omega_1^\omega=\omega_1$ open sets, which are countable unions of basic open sets in $[0,1]^{\omega_1}$).

Firstly we describe the construction of $K_{\alpha+1}$ where $\alpha$ is an even ordinal. We assume that $K_\alpha$ is already constructed and for each $\beta<\alpha$ the following are satisfied:
\begin{enumerate}
    \item if $\beta\in\even$ then we have infinite sets $b_\beta^*\subseteq a_\beta^*\subseteq \omega$ such that $\{ q_n|\alpha: n\in a_\beta^* \}$ is relatively discrete and 
    
$$\overline{\{q_n|\alpha:n\in b_\beta^*\}}\cap\overline{\{q_n|\alpha:n\in a_\beta^*\backslash b_\beta^*\}}\neq \varnothing. $$
 \item if $\beta\in\odd$ then we have infinite sets $b_\beta^i\subseteq a_\beta^i\subseteq\omega$ for $i=0,1$ such that the set $\{q_n|\alpha: n\in a_\beta^i\}$ is relatively discrete and  $$\overline{\{q_n|\alpha:n\in b_\beta^i\}}\cap\overline{\{q_n|\alpha:n\in a_\beta^i\backslash b_\beta^i\}}\neq \varnothing $$
 for $i=0,1$.
\end{enumerate}
Put $U_n^\alpha= \bigcup_{m\in \omega} U_{n,m}^\alpha$. We will say that even step $\alpha$ is non-trivial if
\begin{itemize}
    \item there is $\delta>0$ such that $|\mu_n^\alpha|(U_n^\alpha\cap K_\alpha)>\delta$ for each $n\in\omega$,
    \item $(U_n^\alpha\cap K_\alpha)_{n\in\omega}$ are pairwise disjoint,
    \item $\{q_{l_n^\alpha}: n\in \omega\}$ is relatively discrete in $K_\alpha$,
    \item $\{q_{l_n^\alpha}: n\in \omega\}\cap U_m^\alpha = \varnothing$ for $m\in\omega$.
\end{itemize}
Otherwise we call this step trivial and we put $K_{\alpha+1}=K_\alpha\times \{0\}$ and $q_n|\alpha+1=q_n|\alpha^\frown 0$.

Assume that we are in a non-trivial case.
Apply proposition \ref{metric-functions} for $U_n=U_n^\alpha\cap K_\alpha, \mu_n=\mu_n^\alpha$ to obtain $(f_n^\alpha)_{n\in\omega}\subseteq C_I(K_\alpha)$, infinite $N\subseteq \omega$ and $\varepsilon>0$ such that 
\begin{itemize}
    \item $\supp(f_n^\alpha)\subseteq U_n^\alpha\cap K_\alpha$ for $n\in N$,
    \item $|\int f_n^\alpha d\mu_n^\alpha|>\varepsilon$ for $n\in N$,
    \item $\sum \{|\int f_m^\alpha d\mu_n|: n\neq m, m\in N \}<\varepsilon/3$ for $n\in N$,
    \item $K_\alpha\backslash D((f_n^\alpha)_{n\in N})$ is zero-dimensional.
\end{itemize}
By Lemma \ref{many-strong-extensions}, without loss of generality (by passing to an infinite subset of $N$) we may assume that for all infinite $B\subseteq N$ the extension $K_\alpha(B)$ of $K_\alpha$ by $(f_n^\alpha)_{n\in B}$ is strong and for each $\beta<\alpha$ and $i\in\{*,0,1\}$ we have $$\overline{\{q_n^B|\alpha+1:n\in b_\beta^i\}}\cap\overline{\{q_n^B|\alpha+1:n\in a_\beta^i\backslash b_\beta^i\}}\neq \varnothing, $$
where 
$$q_l^B|\alpha+1= q_l|\alpha^\frown t,  t=\sum_{n\in B} f_n^\alpha(q_l|\alpha),$$
and the closures are taken in $K_\alpha(B)$. 

Let $a_\alpha^*=\{l_n^\alpha: n\in N\}$. Then 
\begin{gather}\tag{$*$}
N=\{n\in \omega: l_n^\alpha \in a_\alpha^* \}.   
\end{gather}
We will show that there is infinite $b_\alpha^*\subseteq a_\alpha^*$ such that  
$$\overline{\{q_n|\alpha:n\in b_\alpha^*\}}\cap\overline{\{q_n|\alpha:n\in a_\alpha^*\backslash b_\alpha^*\}}\neq \varnothing. $$
Suppose otherwise. Then since $K_\alpha$ is a compact metrizable space, for each $X\subseteq a_\alpha^*$ there are disjoint open sets $U_X, V_X$ such that 
$$\overline{\{q_n|\alpha:n\in X\}}\subseteq U_X, \overline{\{q_n|\alpha:n\in a_\alpha^*\backslash X\}}\subseteq V_X,$$
and $U_X, V_X$ are finite unions of members of some fixed countable base in $K_\alpha$. There are uncountably many choices of $X$ and only countably many pairs of such open sets in $K_\alpha$, so for some $X\neq Y$ we have $\{U_X,V_X\}=\{U_Y,V_Y\}$ which is a contradiction. 

Let $b_\alpha^*$ be such that
$$\overline{\{q_n|\alpha:n\in b_\alpha^*\}}\cap\overline{\{q_n|\alpha:n\in a_\alpha^*\backslash b_\alpha^*\}}\neq \varnothing$$
and define 
\begin{gather}\tag{$**$}
B=\{n\in N: l_n^\alpha \in b_\alpha^* \}.   
\end{gather}

To finish the construction at this step we put $K_{\alpha+1}=K_\alpha(B), q_n|\alpha+1= q_n^B|\alpha+1$ and observe that (1) is satisfied for $a_\alpha^*,b_\alpha^*$, because if 
$$x\in \overline{\{q_n|\alpha:n\in b_\alpha^*\}}\cap\overline{\{q_n|\alpha:n\in a_\alpha^*\backslash b_\alpha^*\}},$$
then
$$(x,0)\in \overline{\{q_n|\alpha+1:n\in b_\alpha^*\}}\cap\overline{\{q_n|\alpha+1:n\in a_\alpha^*\backslash b_\alpha^*\}},$$
since $f_n^\alpha(q_k|\alpha)=0$ for all $n\in B$ and $k\in a_\alpha$.

At step $\alpha\in \odd$ we assume that we are given $a_\beta^i, b_\beta^i$ satisfying (1) and (2) from the even step for all $\beta<\alpha$ (where $i=*$ if $\beta\in \odd$ and $i\in \{0,1\}$ if $\beta\in \even$). We call this step non-trivial, if the closures of $\pi_\alpha[U_\alpha]$ and $\pi_\alpha[V_\alpha]$ have non-empty intersection. If the case is non-trivial we use Lemma \ref{lemma-for-odd} (note that Lemma \ref{connected-extension} implies that $K_\alpha$ is connected) to find appropriate $(f_n)_{n\in\omega}\subseteq C_I(K_\alpha), A_i$ and $S_i$ for $i=0,1$. In the same way as in the even step we find $B\subseteq \omega$ such that $K_\alpha(B)$ is a strong extension of $K_\alpha$ and the conditions (1) and (2) are preserved in $K_\alpha(B)$ for $\beta<\alpha$. To finish this step we define $K_{\alpha+1}=K_\alpha(B), a_\alpha^i=S_i, b_\alpha^i=A_i$ and $q_n|\alpha+1=q_n^B|\alpha+1$. Lemma \ref{lemma-for-odd} guarantees that the condition (2) holds at the step $\alpha+1$.

In both cases the density of $Q_{\alpha+1}=\{q_n|\alpha+1:n\in\omega\}$ in $K_{\alpha+1}$ follows from Lemma \ref{extension-separable}.

\begin{proof1}
We will show that the space constructed above satisfies the required conditions.
(1) follows from Corollary \ref{corollary-dimension} and the fact that $[0,1]^k$ is a $k$-dimensional space. $Q$ is a countable dense set in $K$, since each $Q_\alpha$ is dense in $K_\alpha$ for $\alpha<\omega_1$. Connectedness follows from inductive argument using Lemma \ref{connected-extension}.

Let $U_n, l_n, \mu_n$ be as in (4). Let $U_n=\bigcup_{m\in\omega} U_{n,m}\cap K$ where $U_{n,m}$ are basic open sets in $[0,1]^{\omega_1}$. Every $U_{n,m}$ is determined by finitely many coordinates, so there is $\gamma<\omega_1$ such that $\pi_\gamma^{-1}(\pi_\gamma[U_{n,m}])= U_{n,m}$ for $n\in \omega$, where $\pi_\gamma$ is the natural projection from $[0,1]^{\omega_1}$ onto $[0,1]^\gamma$ (so $U_{n,m}$ are determined by first $\gamma$ coordinates). By Lemma \ref{diamond-lemma} there is $\alpha>\gamma, \alpha\in \even$ such that for $n\in\omega$ 
\begin{itemize}
    \item $\mu_n|C(K_\alpha)=\mu_n^\alpha$,
    \item $\pi_\alpha[U_{n,m}]=U_{n,m}^\alpha$,
    \item $l_n=l_n^\alpha$.
\end{itemize}
Let $(f_n^\alpha)_{n\in B}$ be such that in the $\alpha$-th step of construction. Since $(f_n^\alpha)_{n\in B}$ satisfy conditions of Proposition \ref{metric-functions},
functions $f_n=f_n^\alpha\circ\pi_\alpha$ satisfy conditions (a-c). (d) follows from \cite[Lemma 4.6]{few-operators-main} and the fact that $K_{\alpha+1}$ is a strong extension of $K_\alpha$ by $(f_n)_{n\in B}$. By construction we have $$ \overline{\{q_n:n\in b_\alpha^*\}}\cap\overline{\{q_n:n\in a_\alpha^*\backslash b_\alpha^*\}}\neq \varnothing$$
and by $(*)$ and $(**)$ 
$$\{q_n:n\in b_\beta^*\}= \{q_{l_n}:n\in B\},$$
$$ \{q_n:n\in a_\alpha^*\backslash b_\alpha^*\}= \{q_{l_n}:n\in N\backslash B\},$$
which gives (e).

Now we will prove (5). Fix open sets $U,V\subseteq K$ such that $\overline{U}\cap \overline{V}\neq \varnothing$. As $K$ is separable it is c.c.c. so there are open $U'\subseteq U, V'\subseteq V$ which are countable unions of basic open sets such that $\overline{U'}=\overline{U}$ and $\overline{V'}=\overline{V}$ (namely  it is enough to take as $U'$ the union of a maximal antichain of open subsets in $U$, and similarly for $V'$). Without loss of generality we may assume that $U'=U$ and $V'=V$.
Since $U,V$ are countable unions of basic open sets, there is $\gamma<\omega_1$ such that $U,V$ are determined by coordinates less than $\gamma$. Let $\alpha> \gamma, \alpha\in \odd$ be such that $U=U_\alpha\cap K, V=V_\alpha\cap K$. Then $\overline{\pi_\alpha[U]}\cap \overline{\pi_\alpha[V]}$ is nonempty so $\alpha$-th step in construction is nontrivial. By construction we have for $i=0,1$
$$\overline{\{q_n|\beta:n\in b_\alpha^i\}}\cap\overline{\{q_n|\beta:n\in a_\alpha^i\backslash b_\alpha^i\}}\neq \varnothing $$ for all $\beta>\alpha$, so there are $x_i\in \overline{U}\cap \overline{V}$ such that
$$x_i\in \overline{\{q_n:n\in b_\alpha^i\}}\cap\overline{\{q_n:n\in a_\alpha^i\backslash b_\alpha^i\}}.$$
To finish the proof we need only to notice that $x_0\neq x_1$, but this follows form the fact that $a_\alpha^i, b_\alpha^i$ were chosen to satisfy Lemma \ref{lemma-for-odd}(3). 
\end{proof1}

\begin{lemma}\label{borel-functions}
Suppose that $(U_n)_{n\in\omega}$ is a sequence of pairwise disjoint open subsets of a compact Hausdorff space $K$. Let $M,N\subset\omega$ be infinite sets such that $M\cap N$ is finite. Assume that $(f_m)_{m\in M}, (g_n)_{n \in N}\subseteq C_I(K)$ are such that $\supp(f_m)\subseteq U_m, \supp(g_n)\subseteq U_n$ for $m\in M, n\in N$ and the suprema $f_{\sup}=\sup \{f_m: m\in M\}, g_{\sup}= \sup \{g_n:n\in N\}$ exist in $C_I(K)$. Denote $$f=f_{\sup} - \sum_{m\in M} f_m, \ g=g_{\sup} - \sum_{n\in N} g_n.$$
Then $f,g$ are Borel functions with disjoint supports. 
\end{lemma}
\begin{proof}
$f$ and $g$ are Borel functions since they are pointwise sums of countably many continuous functions. Put $D=M\cap N$ and note that since $D$ is finite the function $\sum_{m\in D} f_m$ is continuous. We will show that 
\begin{gather}\tag{$+$}
    \sup\{f_m: m\in M\backslash D \}=\sup\{f_m: m\in M\} - \sum_{m\in D} f_m.
\end{gather}
Let $x\in K$. If $x\in \supp(f_n)$ for some $n\in  M\backslash D$, then $\sum_{m\in D} f_m (x) =0$, so $$(\sup\{f_m: m\in M\} - \sum_{m\in D} f_m)(x)= \sup\{f_m: m\in M\}(x)\geq f_n(x)$$ for every $n\in M\backslash D$. If $x\notin \supp(f_n)$ for every $n\in M\backslash D$, then since $f_n$'s have disjoint supports we get that    
$$(\sup\{f_m: m\in M\} - \sum_{m\in D} f_m)(x)\geq 0= f_n(x)$$ for $n\in M\backslash D$. Hence $$\sup\{f_m: m\in M\} - \sum_{m\in D} f_m\geq f_n$$ for $n\in M\backslash D$ in the lattice $C(K)$. Let $h\in C(K)$ be such that 
$$\sup\{f_m: m\in M\} - \sum_{m\in D} f_m\geq h\geq f_n$$
for $n\in M\backslash D$. Since $f_n$'s have disjoint supports we have 
$$\sup\{f_m: m\in M\} \geq h+\sum_{m\in D} f_m\geq \sum_{m\in M} f_m.$$
But $$\sup\{f_m: m\in M\}(x)= \sum_{m\in M} f_m(x)$$ for $x\in D((f_n)_{n\in M})$, so $$\sup\{f_m: m\in M\} - \sum_{m\in D} f_m = h,$$ because $\sup\{f_m: m\in M\} - \sum_{m\in D} f_m$ and $h$ are continuous functions equal on the set $D((f_n)_{n\in M})$, which is dense in $K$ (cf. Lemma \ref{continuous-sum}). This completes the proof of the equality $(+)$. 

From $(+)$ we get that $$\sup\{f_m: m\in M\backslash D \}- \sum_{m\in M\backslash D} f_m = \sup\{f_m: m\in M\} - \sum_{m\in M} f_m= f.$$
In particular in the definition of $f$ we may replace $M$ with $M\backslash D$ and assume that $M\cap N=\varnothing$. 

We will show that in this case we have $\supp(f_{\sup})\cap \supp(g_{\sup}) = \varnothing$, which will finish the proof since $\supp(f)\subseteq \supp(f_{\sup})$ and $\supp(g)\subseteq \supp(g_{\sup})$ (the inclusions hold because $f\leq f_{\sup}, g\leq g_{\sup}$ and $f,g$ are non-negative).
Firstly we observe that for each $n\in N$ we have $\supp(f_{\sup})\cap \supp(g_n)=\varnothing$. Indeed, if it is not the case, then there is $x\in U_n$ such that $f_{\sup}(x)>0$. Then by the Tietze extension theorem we may find $h\in C_I(K)$ such that $h(x)=0$ and $h|K\backslash U_n= f_{\sup}|K\backslash U_n$. But then $f_{\sup}>h\geq f_m$ for every $m\in M$, which is a contradiction with the fact that $f_{\sup}$ is the supremum of $f_m$'s. Now, in the same way we show that if $\supp(f_{\sup})\cap \supp(g_{\sup}) \neq \varnothing$, then there is $h'$ such that $g_{\sup}>h'>g_n$ for $n\in N$.
\end{proof}

\begin{theorem}\label{theorem-construction}
Assume $\diamondsuit$.
For each $k>0$ there is a compact Hausdorff, separable, connected space $K$ such that $C(K)$ has few operators and $\dim K=k$.
\end{theorem}

\begin{proof}
We will show that if $K$ is the space with properties from Proposition \ref{construction-properties}, then $C(K)$ has few operators. $K$ satisfies Proposition \ref{construction-properties}(5), so by \cite[Theorem 2.7, Lemma 2.8]{few-operators-main} it is enough to show that all operators on $C(K)$ are weak multipliers. 

Assume that there is a bounded linear operator $T:C(K)\rightarrow C(K)$, which is not a weak multiplier. By Theorem \ref{weak-multiplier} there is a pairwise disjoint sequence $(g_n)_{n\in\omega}\subseteq C_I(K)$ and pairwise disjoint open sets $(V_n)_{n\in \omega}$ such that $g_n|V_m=0$ for $n,m\in \omega$ and $|T(g_n)|V_n|>\delta$ for some $\delta>0$. 
    For $n\in \omega$ let $U_n=\supp(g_n)$. Let $g'_n\in C([0,1]^{\omega_1})$ be an extension of $g_n$ and $U_n'=\supp(g_n')$. By Mibu's theorem (see \cite{mibu}) for every $n\in\omega$ there is $\alpha_n<\omega_1$ such that whenever $x,y\in [0,1]^{\omega_1}, x|\alpha_n=y|\alpha_n$, we have $g'_n(x)=g'_n(y)$. Hence $U'_n$ is an open set of the form $W_n\times [0,1]^{\omega_1 \backslash \alpha_n}$, where $W_n$ is an open set in $[0,1]^{\alpha_n}$. Since $\alpha_n$ is countable, $W_n$ is a union of countably many basic open set in $[0,1]^{\alpha_n}$. Thus for every $n\in\omega$ the set $U'_n$ is a union of countably many basic open sets in $[0,1]^{\omega_1}$ and $U_n=U'_n\cap K$ is a union of countably many basic open sets in $K$. 

Let $(l_n)_{n\in\omega}$ for $n\in\omega$ be such that $q_{l_n}\in V_n$ (so in particular $\{q_{l_n}:n\in \omega\}$ is relatively discrete in $K$) and define $\mu_n=T^*(\delta_{q_{l_n}})$. Then $|\int g_n d\mu_n|= |T(g_n)(q_{l_n})|>\delta$. Since $\supp(g_n)\subseteq U_n$ and $\|g_n\|\leq 1$ we get that $|\mu_n|(U_n)\geq |\int g_n d\mu_n|>\delta$. 

By Proposition \ref{construction-properties} for every infinite subset $A\subseteq \omega$ there are infinite sets $B_A\subseteq N_A\subseteq A$, continuous functions $(f_{n,A})_{n\in A}\subseteq C_I(K)$ and $\varepsilon_A$ such that 
\begin{enumerate}[(a)] 
    \item $(f_{n,A})_{n\in A}$ is a sequence of pairwise disjoint functions with $\supp(f_{n,A})\subseteq U_n$ for $n\in A$,
    \item $|\int f_{n,A} d\mu_n|>\varepsilon_A$ for $n\in B_A$,
    \item $\sum \{|\int f_{m,A} d\mu_n|: n\neq m, m\in B_A \}<\varepsilon_A/3$ for $n\in N_A$,
    \item $\{f_{n,A}:n\in B_A\}$ has its supremum in the lattice $C(K)$,
    \item  $\overline{\{q_{l_n}: n\in B_A\}}\cap \overline{\{q_{l_n}: n\in N_A\backslash B_A\}}\neq \varnothing.$
    \end{enumerate}

Put $f_A=\sup\{f_{n,A}: n\in B_A\}-\sum_{m\in B_A} f_{m,A}$. We will show that there is an infinite set $M\subseteq \omega$ such that
\begin{gather}\tag{$++$}
\int f_M d\mu_n=0.    
\end{gather}

Suppose this is not the case. Let $\{M_{\xi}: \xi<\omega_1\}$ be a family of infinite subsets of $\omega$ such that for $\xi\neq \xi'$ the set $M_\xi\cap M_{\xi'}$ is finite. Assume ($++$) does not hold for every $M_\xi$. Then there is $n\in \omega$ such that 
$$\int f_{M_\xi}d\mu_n\neq 0$$ for uncountably many $\xi$'s. By Lemma \ref{borel-functions} $f_{M_\xi}, f_{M_{\xi'}}$ have disjoint supports for $\xi\neq \xi'$, so in particular there is an uncountable family of non-null (with respect to $\mu_n$) Borel sets in $K$, which is a contradiction. 

Put $f_n=f_{n,M}, \varepsilon=\varepsilon_M$, $B=B_M$ and $N=N_M$. Let $f=\sup\{f_n: n\in B\}$.
By (b), (c), ($++$) and the definition of $\mu_n$ we get that for $n\in B$
\begin{gather*}
    |T(f)(q_{l_n})|= |\int fd\mu_n|= |\int f_nd\mu_n+\int\sum_{m\in B\backslash \{n\}} f_m|\geq \\
    |\int f_nd\mu_n|- |\int\sum_{m\in B\backslash \{n\}} f_m|\geq \varepsilon - \varepsilon/3=2\varepsilon/3.
\end{gather*}
For $n\in N\backslash B$ (c) gives 
$$|T(f)(q_{l_n})|= |\int \sum_{m\in B} f_md\mu_n|<\varepsilon/3.$$
As $T(f)$ is a continuous function on $K$ we obtain that $$\overline{\{q_{l_n}:n\in B\}}\cap \overline{\{q_{l_n}:n\in N\backslash B\}}=\varnothing,$$
which contradicts (e).
\end{proof}

\begin{theorem}\label{main}
Assume $\diamondsuit$. Then for every $k\in \omega\cup \{\infty\}$ there is a compact Hausdorff space $K$ such that $\dim(K)=k$ and whenever $C(K)\sim C(L)$, $\dim(L)=k$.
\end{theorem}
\begin{proof}
For $k=0$ every finite space $K$ works. If $k>0$, then the space from Theorem \ref{theorem-construction} satisfies the required property by Corollary \ref{corollary-dimension-equal}.
\end{proof}

\section{Remarks and questions}

The first natural question concerning our results is whether Theorem \ref{main} is true without any additional assumption. 
\begin{question}\label{question1}
Let $k\in \omega\backslash \{0\}$. 
Is there (in {\sf ZFC}) a compact Hausdorff space $K$ such that $\dim(K)=k$ and whenever $C(K)\sim C(L)$, $\dim(L)=k$?
\end{question}
In the light of Theorem \ref{few-operators-modulo-finite} to show that the Question \ref{question1} has positive answer it would be enough to prove that the following question has positive answer. 
\begin{question}
Let $k\in \omega\backslash \{0\}$. 
Is there (in {\sf ZFC}) a compact, separable, connected Hausdorff space $K$ such that $\dim K=k$ and $C(K)$ has few operators? 
\end{question}

The original construction of a Banach space $C(K)$ where all the operators are weak multipliers was carried out in {\sf ZFC} (\cite{few-operators-main}). In this construction we set all sequences of pairwise disjoint continuous functions on $[0,1]^{\mathfrak c}$ into a sequence of length $\mathfrak c$, and the choice of the strong extension at $\alpha$-th step depends on the $\alpha$-th sequence of functions. Later, in order to prove that $K$ satisfies the required conditions, we look at any sequence $(\mu_n)_{n\in\omega}$ of Radon measures on $K$ and show that we can find sequences of continuous functions satisfying properties (a-e) from Proposition \ref{construction-properties}. However, in this approach we may obtain an infinite-dimensional space, since used strong extensions may increase the dimension. One can try to proceed in a similar way, by applying only those extensions that preserve the dimension. The problem is that we do not know, whether the extension by the sequence of functions given at some step changes the dimension, since it depends on the earlier steps (i.e. it depends on the bookkeeping of sequences of continuous functions on $[0,1]^{\mathfrak c}$). Consequently, there may be a sequence of measures on the final space, for which every suitable sequence of functions appears at a step, in which using the extension would increase the dimension. 

Although the main reason to use the diamond principle is the guessing of measures in Lemma \ref{diamond-lemma}, we also needed the continuum hypothesis to ensure that all intermediate spaces from our construction are metrizable. At that point we used the fact that for every non-zero Radon measure on metrizable compact space there is a zero-dimensional $G_\delta$ compact subset of non-zero measure (Theorem \ref{measure-zero-dim}). In the light of this theorem the following problem seems to be interesting. 

\begin{problem}
Describe the class of compact Hausdorff spaces $K$ such that for every non-zero Radon measure $\mu$ on $K$ there is a zero-dimensional compact subset $L\subseteq K$ such that $\mu(L)\neq 0$.
\end{problem}

Assume that $K$ is such that $C(K)$ has few operators. Then by \cite[Proposition 4.8]{schlackow} there is a space $L$ such that $C(K)\sim C(L)$, but $C(L)$ does not have few operators. However, by Theorem \ref{few-operators-modulo-finite} the topology of $L$ is very close to $K$, at least if we assume that $K$ is separable and connected. 

\begin{question}
Suppose that $K$ is a compact Hausdorff space such that every operator $T:C(K)\rightarrow C(K)$ is a weak multiplier and $C(L)\sim C(K)$ for some compact Hausdorff space. Is it true that $K$ and $L$ are homeomorphic modulo finitely many points in the sense of Theorem \ref{few-operators-modulo-finite}?
\end{question}

One may also ask, what properties $K$ should have to satisfy Theorem \ref{main}. There are known examples of ``nice" spaces $K$ such that if $C(K)\sim C(L)$, then $L$ is not zero-dimensional. For instance Avil\'es and Koszmider showed that there is such a space which is quasi Radon-Nikod\'ym (\cite{AK}) and Plebanek gave a consistent example of such a space which is a Corson copmact (\cite{plebanek-corson2}).

\section*{Acknowledgements}
The author would like to thank his PhD supervisor Professor Piotr Koszmider for introducing to the topic, constant help and many valuable
suggestions. 
\bibliographystyle{amsplain}
\bibliography{bibliography}
\end{document}